\documentclass[article]{siamart1116}
\usepackage{amsmath}
\usepackage{float}
\usepackage{diagbox}
\usepackage{bbold}
\usepackage{hyperref}
\usepackage{enumitem}

\newtheorem{thm}{Theorem}

\newtheorem{rmk}{Remark}

\newcommand{\rd}{\,\mathrm{d}}
\newcommand{\rev}[1]{\textcolor{blue}{#1}}

\newcommand{\bhat}{\widehat{\vb} }
\newcommand{\Chat}{\widehat{\vc} }

\newcommand{\bdot}{\dot{\vb}}
\newcommand{\Cdot}{\dot{\vc}}

\newcommand{\be}{\begin{equation}}
\newcommand{\ee}{\end{equation}}

\newcommand{\Dt}{\Delta t}

\newcommand{\cR}{{\cal R}}

\newcommand{\dt}{\Delta t}

\newcommand{\m}[1]{\mathbf{#1}}
\newcommand{\mR}{\m{R}}
\newcommand{\mP}{\m{P}}
\newcommand{\mW}{\m{W}}

\newcommand{\mD}{\m{D}}
\newcommand{\mDdot}{\m{\dot{D}}}
\newcommand{\mAdot}{\m{\dot{A}}}
\newcommand{\mA}{\m{A}}
\newcommand{\mAh}{\m{\hat{A}}}
\renewcommand{\v}[1]{\boldsymbol{#1}}

\newcommand{\vb}{\v{b}}
\newcommand{\vc}{\v{c}}

\newcommand{\ve}{\v{e}}

\newcommand{\sspcoef}{\mathcal{C}}

\newcommand{\DtFE}{\Dt_{\textup{FE}}}

\renewcommand{\v}[1]{\mathbf{#1}}

\title{High order  strong stability preserving multi-derivative implicit and
IMEX Runge--Kutta methods with  asymptotic preserving properties}

\author{%
Sigal Gottlieb\thanks{Mathematics Department, University of Massachusetts Dartmouth, North Dartmouth, MA 02747. Email: sgottlieb@umassd.edu.
SG's research was supported in part by AFOSR Grant No. FA9550-18-1-0383.} \and
Zachary J. Grant\thanks{Multiscale Methods and Dynamics Group, Mathematics in Computation subsection, Oak Ridge National Laboratory, Oak Ridge, TN 37830. Email: grantzj@ornl.gov.
This manuscript has been authored by UT-Battelle, LLC, under contract DE-AC05-00OR22725 with the US Department of Energy (DOE). The US government retains and the publisher, by accepting the article for publication, acknowledges that the US government retains a nonexclusive, paid-up, irrevocable, worldwide license to publish or reproduce the published form of this manuscript, or allow others to do so, for US government purposes. DOE will provide public access to these results of federally sponsored research in accordance with the DOE Public Access Plan (http://energy.gov/downloads/doe-public-access-plan).}
\and
Jingwei Hu\thanks{Department of Mathematics, Purdue University, West Lafayette, IN 47907. Email: jingweihu@purdue.edu. JH's research was supported in part by NSF CAREER grant DMS-1654152.} \and
Ruiwen Shu\thanks{Department of Mathematics, University of Maryland, College Park, MD 20742. Email: rshu@cscamm.umd.edu.}
}

\begin{document}
\maketitle


\bibliographystyle{siam}

\begin{abstract}  
In this work we present a class of high order  unconditionally strong stability preserving (SSP)
 implicit two-derivative Runge--Kutta schemes, and SSP implicit-explicit (IMEX) multi-derivative
 Runge--Kutta schemes where the time-step restriction is  independent of the stiff term. 
 The unconditional SSP property for a method of order $p>2$ is unique among SSP methods, and   
depends  on a backward-in-time assumption on the derivative of the operator. 
We show that this backward derivative condition
is satisfied in many relevant cases where SSP IMEX schemes are desired. 
 We devise unconditionally SSP implicit Runge--Kutta schemes of order up to $p=4$,
 and IMEX Runge--Kutta schemes of order up to $p=3$. For the multi-derivative IMEX schemes, we also derive
 and present the order conditions, which have not appeared previously.
 The unconditional SSP condition ensures that these methods are positivity preserving, and 
we present sufficient conditions under which such methods are also asymptotic preserving 
when applied to a range of problems, including a hyperbolic relaxation system, the Broadwell model, 
and the Bhatnagar-Gross-Krook (BGK) kinetic equation. 
We present numerical results to support the theoretical results, on a variety of problems. 
\end{abstract}

\section{Introduction}   \label{sec:intro}

Explicit strong stability preserving (SSP) Runge--Kutta methods were first developed for use  with 
total variation diminishing spatial discretizations for hyperbolic conservation laws with discontinuous solutions
\cite{Shu1988a,Shu1988b}. They have  proven useful in  a wide variety of problems where we need to evolve an ODE,
as they preserve any convex functional property satisfied by the forward-Euler method, while giving higher order
solutions.
Given a system of ODEs, generally resulting from a spatial discretization of a PDE, of the form
\begin{eqnarray} \label{ODE1}
u_t = G(u)
\end{eqnarray}
that satisfies some forward Euler condition 
\begin{eqnarray} \label{FEcond}
\hspace{-0.7in} {\mbox{\bf Forward Euler condition:} } & \nonumber \\
&  \| u + \dt G(u)\| \leq \| u \|  \; \; \;   \mbox{ for all } \;\;  \dt \leq  \DtFE, 
  \end{eqnarray}
where $\| \cdot \|$ is some convex functional (e.g. positivity).
In practice, we don't want to use Euler's method. Instead,
we desire a  higher order method that preserves the forward Euler condition,
perhaps under a modified time-step restriction $\dt \leq \sspcoef \DtFE$.
Higher order methods that can be written as convex combinations of forward Euler steps
with $\sspcoef >0$ will preserve the forward Euler condition, and are called SSP.
The value $\sspcoef $ is called the SSP coefficient, and we generally want to devise
methods that have a large $\sspcoef $.

When concerned with linear stability properties, we turn to implicit methods, 
or to implicit-explicit methods, to alleviate the time-step restriction. 
When considering the more strict SSP property, even implicit methods suffer from a step-size restriction 
 that is quite severe: the SSP coefficient is usually bounded by twice the number of stages for a Runge--Kutta 
 method \cite{KetchDW}. This is true for all implicit methods that have been tested: Runge--Kutta, multistep methods,
 and general linear methods. However, by using a second operator $\tilde{G}$ that approximates $G$ and
satisfies a downwind condition
\begin{eqnarray} \label{DWcond}
\hspace{-1in} \mbox{\bf Downwind condition:}&  \nonumber \\
 &  \| u - \dt \tilde{G}(u)\| \leq \| u \|  \; \; \;   \mbox{ for all } \;\;  \dt \leq \; \DtFE, 
 \end{eqnarray}
Ketcheson found a family of  implicit second order methods that are 
unconditionally SSP  \cite{KetchDW}. 
 
 In \cite{SD,TS} the SSP properties of  multi-derivative Runge--Kutta methods were studied. 
 For such methods, in addition to the forward Euler condition \eqref{FEcond},
we need some condition on the second derivative $\dot{G} = \frac{d G}{dt} = G' G $. 
One candidate was a second derivative condition  \cite{SD}:\\
 \noindent{\bf Second derivative condition:}  
\begin{equation} \label{SDcond}
 \| u + \dt^2 \dot{G}(u)\| \leq \| u \|  \; \; \;   \mbox{ for all } \;\;  \dt^2 \leq \tilde{k} \; \DtFE^2, 
\end{equation}
where $\tilde{k}>0$. 
The other possibility was a Taylor series condition \cite{TS}:\\
\noindent{\bf Taylor series condition:}
\begin{equation} \label{TScond}
 \| u +  \dt G(u) + \frac{1}{2}  \dt^2 \dot{G}(u)\| \leq \| u \|  \; \; \;  
 \mbox{ for all } \;\;  \dt \leq \hat{k} \; \DtFE, 
 \end{equation}
 where $\hat{k}>0$.
 \rev{Previously, explicit SSP two-derivative methods were developed that preserved the 
 forward Euler \eqref{FEcond} and second derivative \eqref{SDcond} conditions \cite{SD} or the 
 forward Euler  \eqref{FEcond} and Taylor series  \eqref{TScond} conditions \cite{TS}.
 However, unconditionally implicit methods that preserve the forward Euler condition \eqref{FEcond} 
cannot exist \cite{GST01}. Furthermore, the proof in \cite{GST01} can be easily applied to the
two-derivative case, to show  that there are no unconditionally implicit methods that preserve 
 \eqref{FEcond} and  \eqref{SDcond}, or  \eqref{FEcond} and  \eqref{TScond} (see Appendix A). 
 This leads us to consider the backward derivative condition as an alternative 
 to  \eqref{SDcond} and  \eqref{TScond}.}

To obtain  unconditional SSP methods, we consider in this work a new condition on the second derivative: 
\begin{eqnarray} \label{BDcond}
\hspace{-1.in} {\mbox{\bf Backward derivative condition:} } & \nonumber \\
& \hspace{-.9in}\rev{ \| u - \dt^2 \dot{G}(u)\| \leq \| u \| } \; \; \;   \mbox{ for all } \;\;  \dt^2 \leq \dot{k} \; \DtFE^2, 
 \end{eqnarray}
 for some $\dot{k}>0$. 
 Under this condition, we {\em require} negative coefficients on the derivative, and in this way are able to 
 obtain unconditionally SSP two-derivative Runge--Kutta methods.
In Subsection \ref{sec:MDRKSSP}, we show the conditions under which 
such an implicit two derivative method is unconditionally SSP, in the sense that it preserves the 
strong stability condition \eqref{BDcond} for any positive time-step $\dt$.
In Subsection \ref{sec:New_iMDRK} we proceed to present unconditionally SSP methods of this type 
of order up to $p=4$.

After establishing that unconditionally SSP  implicit two derivative Runge--Kutta 
methods of up to fourth order exist in Subsection \ref{sec:New_iMDRK}, we proceed to expand the theory in Subsection
\ref{sec:MDRKSSP} to implicit-explicit multi-derivative Runge--Kutta 
methods. We devise IMEX methods that are SSP under a time-step restriction resulting only from the operator treated
explicitly. We consider equations of the type
 \begin{eqnarray} \label{IMEXeqn}
 u_t= F(u) + G(u),
 \end{eqnarray}
 where $F$ and $G$ satisfy a forward Euler condition, and $\dot{G}$ satisfies a 
 backward derivative condition. Here, the condition on $F$  requires a reasonable size time-step, 
 but the condition on $G$ requires an inconveniently small time-step.
To alleviate this restriction, we present the multi-derivative IMEX approach in 
Section \ref{sec:2D-IMEX}, and give
sufficient  conditions under which we can ensure the method is SSP
under a time-step that depends only on $F$.  
We then derive the order conditions for multi-derivative IMEX methods. 
In Subsection  \ref{sec:new_methods} we present our new second and third order methods. 
A rich area of applications is described in Subsection \ref{sec:models}, 
where the backward derivative condition appears throughout.

One property that is desired in the problems presented in Subsection \ref{sec:models} is positivity for time-steps that 
depend only on $F$. Being SSP,  the methods in Subsection \ref{sec:new_methods} automatically 
preserve this positivity property. Furthermore, our methods satisfy an additional condition: 
that either or both $G$ and $\dot{G}$ appear in each stage. 
This condition is not needed for SSP (or, equivalently, positivity), but it is valuable
for an additional property that is of interest: they are asymptotic preserving, as we prove in Subsection \ref{APIMEX}.

Taken together, we present unconditionally SSP -- and thus positivity preserving --
methods: both implicit two-derivative Runge--Kutta methods
and IMEX multi-derivative Runge--Kutta methods, where the time-step restriction comes from the explicit part.
The IMEX methods are also asymptotic preserving, which is valuable for the problems in Subsection \ref{sec:models}.
These results are  significant, in that unconditionally SSP methods of order $p>1$ are rare. 
We are limited only by the fact that  the function and its derivative must satisfy 
a forward Euler  \eqref{FEcond}  and backward derivative \eqref{BDcond} conditions, respectively. 
While the forward Euler condition \eqref{FEcond} seems standard, 
the backward derivative condition \eqref{BDcond}  seems, at first glance, to be a bit unusual. 
However, there is a similarity between it and the downwinding condition \eqref{DWcond}.
Furthermore, it turns out that it is a natural condition, and quite useful for a variety of problems, 
as we show in Subsection \ref{sec:models}.

%
%

 \section{SSP implicit two-derivative Runge--Kutta methods}
 
In this section, we consider  two-derivative Runge--Kutta methods
 for the ODE
 \[u_t = G(u).\]
 As discussed in \cite{SD}, the two-derivative Runge--Kutta method 
 can be written in the Butcher form
 \begin{subequations} \label{MDRKButcher}
\begin{eqnarray}
&&u^{(i)}  =   u^n + \dt  \sum_{j=1}^{i} a_{ij} G(u^{(j)}) 
+  \dt^2 \sum_{j=1}^{i} \dot{a}_{ij} \dot{G}(u^{(j)}), \quad  \; \; i=1, . . . , s,   \\
&&u^{n+1} =  u^{(s)}.
 \end{eqnarray}
 \end{subequations}
In matrix form, this becomes
 \begin{eqnarray} \label{iRKmatrixButcher}
 U = \ve u^n   +  \dt \mA G(U) +  \dt^2 \mAdot  \dot{G}(U),
 \end{eqnarray}
 where $\ve$ is a vector of ones.
 
We proceed to define the order conditions of such a method in the next subsection.

 \subsection{Formulating the order conditions}
 
 Given the Butcher form \eqref{iRKmatrixButcher}, 
the vectors   $\vb$ and $\dot{\vb}$ are given by the last row of  $\mA$ and $\dot{\mA}$, respectively.
The vectors $\vc=\mA \ve$ and $\Cdot=\dot{\mA} \ve$ define the time-levels at which the stages are happening;
these values are known as the abscissas.   The order conditions for methods of this form are
given in \cite{SD} up to sixth order. We repeat them here up to fourth order.

\smallskip

 \begin{tabular}{ll}
$ p=1$: & $\vb^T \ve=1, $\\ \\
$ p=2$:  & $ \vb^T \vc+ \dot{\vb}^T \ve=\frac{1}{2},$ \\ \\
$ p=3$: & $ \vb^T \vc^2+2 \dot{\vb}^T \vc= \frac{1}{3} $, \; \; \; \;  
$\vb^T A \vc+ \vb^T  \dot{\vc}+ \dot{\vb}^T \vc= \frac{1}{6}, $ \\ \\
 $p=4$: & $\vb^T \vc^3 +3 \dot{\vb}^T \vc^2 =  \frac{1}{4}$, \; \; \;   
 $ \vb^T \vc A \vc+ \vb^T \vc \dot{\vc} +\dot{\vb}^T \vc^2 +  \dot{\vb}^T A \vc+  \dot{\vb}^T \dot{\vc}= \frac{1}{8}, $ \\ \\
	& $\vb^T A \vc^2 +2 \vb^T \dot{A} \vc+  \dot{\vb}^T \vc^2 = \frac{1}{12}, $ \\ \\
	& $\vb^T A^2 \vc+ \vb^T A \dot{\vc}+ \vb^T \dot{A} \vc+ \dot{\vb}^T A \vc+  \dot{\vb}^T \dot{\vc} = \frac{1}{24}. $ \\
 \end{tabular}
 \bigskip

\subsection{Strong stability preserving properties} \label{sec:MDRKSSP}
\rev{To ensure that a method of the form  \eqref{MDRKButcher}  does not 
result in  an SSP time-step restriction, } we write the
method in a special Shu-Osher form with only implicit computations
\begin{subequations} \label{MDRK}
\begin{eqnarray}
&&u^{(i)}  =  r_i u^n + \sum_{j=1}^{i-1} p_{ij} u^{(j)}   +
\dt d_{ii} G(u^{(i)}) +   \dt^2 \dot{d}_{ii} \dot{G}(u^{(i)}), \quad  \; \; i=1, . . . , s,   \\
&&u^{n+1} =  u^{(s)}.
 \end{eqnarray}
 \end{subequations}
\rev{ This form ensures that only implicit evaluations of $G$ and $\dot{G}$
are present, so that we do not have a time-step restriction due to a 
forward Euler, second derivative, or Taylor series term.
The form \eqref{MDRK} ensures that any explicit terms in the  method \eqref{MDRKButcher} 
enter only after they were introduced implicitly in a prior stage.
This is a necessary (but not sufficient)   condition so that an SSP time-step restriction will not occur
\cite{GKS11}}.

In matrix form, this becomes
 \begin{eqnarray} \label{iRKmatrix}
 U = \mR \ve u^n   +  \mP U + \dt \mD G(U) +  \dt^2 \mDdot  \dot{G}(U),
 \end{eqnarray}
 where  $ \mP$ and $\mR = I - \mP $ are $s \times s$ matrices, $r_i$ are the $i$th row sum of $\mR$,  
 and $\mD$ and $\mDdot$ are $s \times s$ diagonal matrices. 
 The numerical solution $u^{n+1} $ is then given by the final element of the vector $U$.
 Note that the relationship between the Butcher form \eqref{iRKmatrixButcher}
 and the Shu-Osher form \eqref{iRKmatrix} is given by 
 \[ \mA = \mR^{-1} \mD,  \; \; \;  \dot{\mA} = \mR^{-1} \dot{\mD}.\]
 \rev{Note that given a method of the form \eqref{iRKmatrixButcher},   it is not always possible to select some matrix of coefficients $\mR$ and thus obtain 
 matrices $\mP$, $ \mD$ and $\dot{\mD}$ where the matrices $\mD$ and $\dot{\mD}$ are
 diagonal. (However, if $\mA$ has only nonzero elements 
 on the diagonal then it is possible). 
On the other hand, is always possible to start from a two derivative method of the form 
 \eqref{iRKmatrix} and write it in the form \eqref{iRKmatrixButcher}.} 

 A method of the  form  \eqref{MDRK}  will be unconditionally SSP under the following conditions:
\begin{thm}  \label{thm:SSPMDRK} 
Let the operators $G$ and $\dot{G} $  satisfy 
 the forward Euler condition
 \[ \| u + \dt G(u)\| \leq \| u \|  \; \; \;   \mbox{ for all } \;\;  \dt \leq  \DtFE\]
 and the backward derivative condition
\[  \| u - \dt^2 \dot{G}(u)\| \leq \| u \|  \; \; \;   \mbox{ for all } \;\;  \dt^2 \leq \dot{k} \; \DtFE^2,\]
for some $\DtFE>0$ and $\dot{k}>0$, and 
 for  some convex functional $\| \cdot \|$.
A method given by  \eqref{iRKmatrix} which  satisfies the  conditions
\begin{eqnarray} \label{coef}
\mR \ve  \geq  0, \; \; \; \;  \mP  \geq  0, \; \; \;  \mD \geq  0, \; \; \; \; \dot{\mD} \leq  0, 
 \end{eqnarray}
 \rev{(where the inequalities are understood componentwise),}
will preserve the strong stability property
\[  \|u^{n+1} \| \leq \| u^n \|  \]
for any positive time-step $\dt >0$.
\end{thm}

\begin{proof}
\rev{The first stage of the method is given by
\[
 u^{(1)}    =  u^n +   \dt d_{11}  G(u^{(1)})  +  \dt^2  \dot{d}_{11}  \dot{G}(u^{(1)}) .
\]
 Using the forward Euler and backward derivative conditions, we can show that 
 $\| u^{(1)} \| \leq \| u^n \|$, whenever $d_{11} \geq 0$ and $\dot{d}_{11} \leq 0$.
 To see this add $(\alpha + \beta ) u^{(1)}$ to both sides and rearrange
 \begin{align*}
 u^{(1)} &=  \frac{u^n}{1+\alpha + \beta}  
 +   \frac{\alpha}{1+\alpha + \beta}  \left(  u^{(1)} +  \frac{1}{\alpha} \dt  {d}_{11}  {G}(u^{(1)}) \right)  \\
&  +  \;  \frac{\beta}{1+\alpha + \beta}  \left(  u^{(1)} -  \frac{1}{\beta} \dt^2  |\dot{d}_{11}|  \dot{G}(u^{(1)}) \right) .
  \end{align*}
 Assuming that $\alpha \geq 0$ and $\beta \geq 0$ we have (from the forward Euler condition and backward derivative condition) 
\[ \| u^{(1)} \| 
   \leq  \frac{1}{1+\alpha + \beta}\left\|  u^n  \right\|  +\frac{\alpha}{1+\alpha + \beta}   \left\|   u^{(1)}  \right\| 
 +  \frac{\beta}{1+\alpha + \beta}  \left\|    u^{(1)}  \right\|, 
\]
 hence
\[  \| u^{(1)} \| \leq  \| u^n\|,\]
  for any $\dt$ such that 
  $\frac{1}{\alpha} \dt d_{11} \leq  \DtFE$ and $  \frac{1}{\beta}  |\dot{d}_{11} | \dt^2 \leq \dot{k} \DtFE^2$. 
Since we can choose $\alpha$ and $\beta$ to be arbitrarily large then this is true for any $\dt$.
 }
 
 \rev{Each subsequent stage of the method is given by
\begin{equation*}
u^{(i)}     = \left( r_i u^n + \sum_{j=1}^{i-1} p_{ij} u^{(j)} \right) 
 +     \dt d_{ii} G(u^{(i)})  + \dt^2 \dot{d}_{ii} \dot{G}(u^{(i)}) ,
\end{equation*}
where we  can now assume that $  \| u^{(j)} \| \leq  \| u^n\|,$ for all $j < i$.
The explicitly computed terms are 
\begin{align*}
 \| u_e^{(i)} \| & =   \left\| r_i u^n + \sum_{j=1}^{i-1} p_{ij} u^{(j)} \right\|  
  \leq    \| r_i u^n\|  +  \left\| \sum_{j=1}^{i-1} p_{ij} u^{(j)} \right\|  \\
 & \leq    r_i  \|u^n\|  +  \sum_{j=1}^{i-1} p_{ij}  \| u^{(j)}\|  
\leq \left( r_i +  \sum_{j=1}^{i-1} p_{ij} \right) \|u^n\|  \\
  & = \|u^n\|.
 \end{align*}}\rev{ due to the  non-negativity of $r_i$ and $p_{ij}$, and the fact that they sum to one.
Note that this condition is  independent of  $\dt$. 
Finally we write each stage as 
\begin{equation*}
 u^{(i)}    =  u_e^{(i)} +   \dt d_{ii}  G(u^{(i)})  +  \dt^2  \dot{d}_{ii}  \dot{G}(u^{(i)}) ,
\end{equation*}
and use the same argument as for the first stage above to show that 
$ \| u^{(i)} \| \leq \| u_e^{(i)} \|  \leq \| u^n\| $ under any  time-step $ \dt$,
provided only that $d_{ii} \geq 0$ and $\dot{d}_{ii} \leq 0$.}
\end{proof}

\smallskip

\subsection{New  SSP implicit two-derivative Runge--Kutta methods up to order $p=4$\label{sec:New_iMDRK}} 
We found second,  third, and fourth order methods that satisfy the conditions above,
and are unconditionally SSP.

\noindent{\bf Second order} The one-stage, second order 
method is simply the implicit Taylor series method
\[ u^{n+1} = u^n + \dt G(u^{n+1}) - \frac{1}{2} \dt^2 \dot{G}(u^{n+1}) .\]

\noindent{\bf Third order}
A two-stage, third order unconditionally SSP implicit two-derivative Runge--Kutta method
is given by the Shu-Osher coefficients
\[ \mD = \left[ \begin{array}{cc}
0 & 0 \\ 0 & 1 \\
\end{array}\right], \; \; \; 
\mDdot = \left[  \begin{array}{rr}
-\frac{1}{6}& 0 \\ 0 & -\frac{1}{3}  \\
\end{array} \right], \; \; \; 
\mP= \left[  \begin{array}{rr}
0 & 0 \\ 1  & 0 \\
\end{array} \right], \; \; 
\mR \ve = \left[  \begin{array}{c}
1 \\
0 \\
\end{array} \right],
\]
and the Butcher coefficients
\[ \mA = \left[  \begin{array}{cc}
0 & 0 \\ 0 & 1 \\
\end{array} \right], \; \; \; 
\mAdot = \left[  \begin{array}{rr}
  - \frac{1}{6}    &            0 \\
  - \frac{1}{6} &    - \frac{1}{3} \\
\end{array} \right].
\]

\noindent{\bf Fourth order} A five-stage, fourth order unconditionally 
SSP implicit two-derivative Runge--Kutta method
is given by the Shu-Osher coefficients
\[ diag(\mD) = \left[  \begin{array}{c}
   0.660949255604937\\
   0.242201390400848\\
   1.137542996287740\\
   0.191388711018110\\
   0.625266691721946\\
 \end{array} \right], \; \; \;
diag(\mDdot) = \left[  \begin{array}{c}
-0.177750705279127 \\         
-0.354733903778084 \\
-0.403963513682271  \\
-0.161628266349058 \\
 -0.218859021269943 \\
 \end{array} \right],
 \]
 \[P=\left[  \begin{array}{ccccc}
   0                   		 & 0                  & 0         &          0      &     \hspace{.2in}        0 \\
   1                   		 & 0                  &  0         &          0     &         \hspace{.2in}         0 \\
   0.084036809261019 &  0.915963190738981     &              0         &          0     &       \hspace{.2in}           0 \\
   0.001511648458457  &    0  & 0.090254853867587  &                 0          &     \hspace{.2in}        0 \\
   0          &       0 &  0 & 1 &  \hspace{.2in}      0 \\
 \end{array}  \right], \]
 \[  \mR \ve = \left[    1, 0, 0,   0.908233497673956, 0 \right]^T .\]
 And Butcher coefficients
 \[ a_{ii} = d_{ii}, \; \; a_{12} = a_{13} = a_{11}, \; \; \; a_{14} = a_{15} =   0.060653001401867,\]
  \[ a_{23} = 0.221847558352979,  \;  a_{24}=a_{25} = 0.020022818960029, \]
\[  a_{34}=a_{35} = 0.102668776898047, \; \; a_{45} = a_{44},\] 
 and 
 \[ \dot{a}_{ii} = \dot{d}_{ii} , \; \;   \dot{a}_{12} = a_{13} = \dot{a}_{11}, \; \; \; \dot{a}_{14} = \dot{a}_{15} = -0.016311560509453,  \]
 \[ \dot{a}_{23} = - 0.324923198367868, \; \; \dot{a}_{24}= \dot{a}_{25} = - 0.029325895786881, \]
 \[  \dot{a}_{34}=\dot{a}_{35} =  - 0.036459667895230, \; \; \dot{a}_{45} = \dot{a}_{44}.\] 
 
 We were unable to find any fifth order methods that satisfy the conditions in Theorem
 \ref{thm:SSPMDRK}.
 
\subsection{Numerical tests}

 \begin{figure}
 \begin{center}
        \includegraphics[width=.49\textwidth]{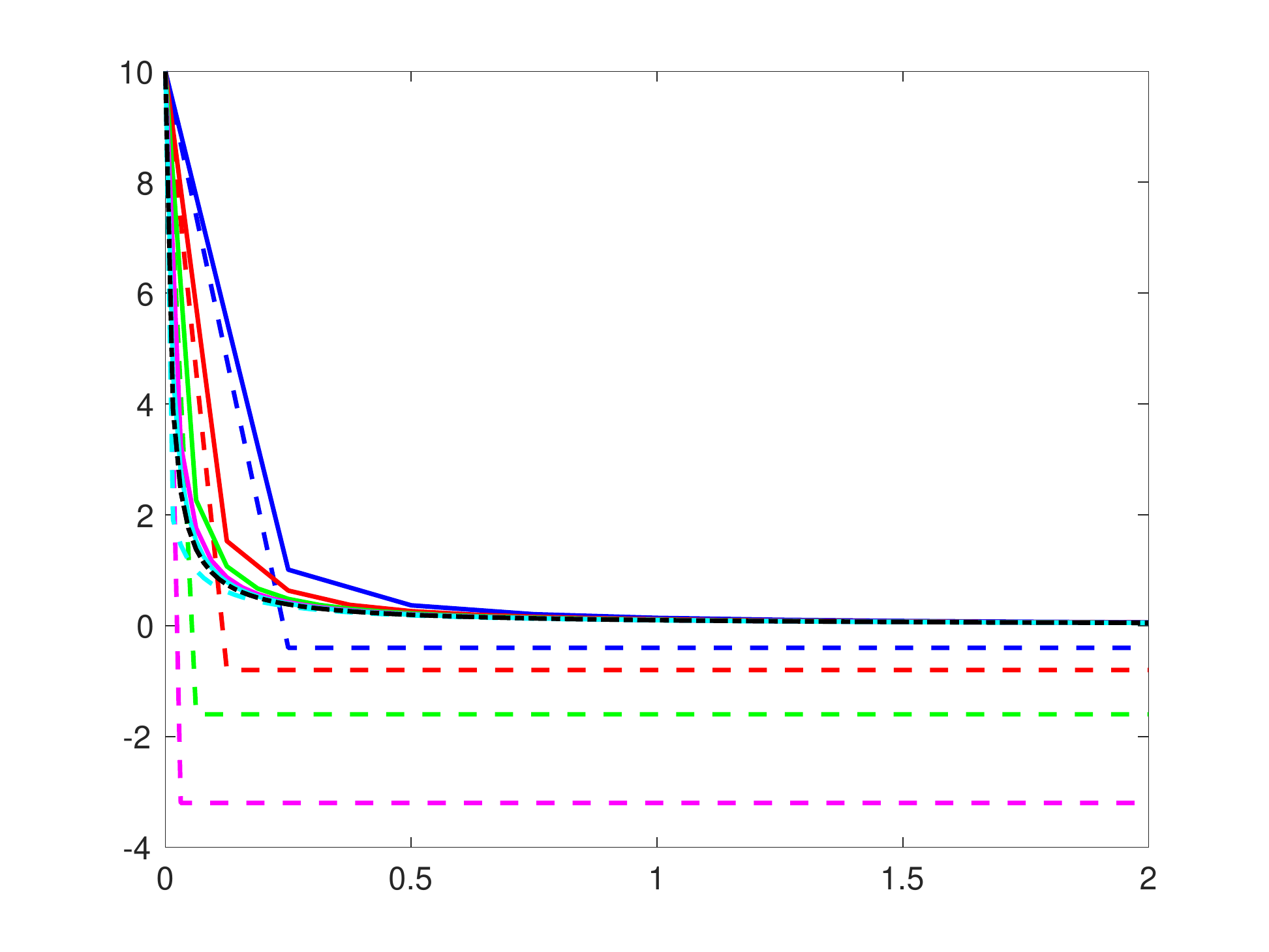} 
        \includegraphics[width=.49\textwidth]{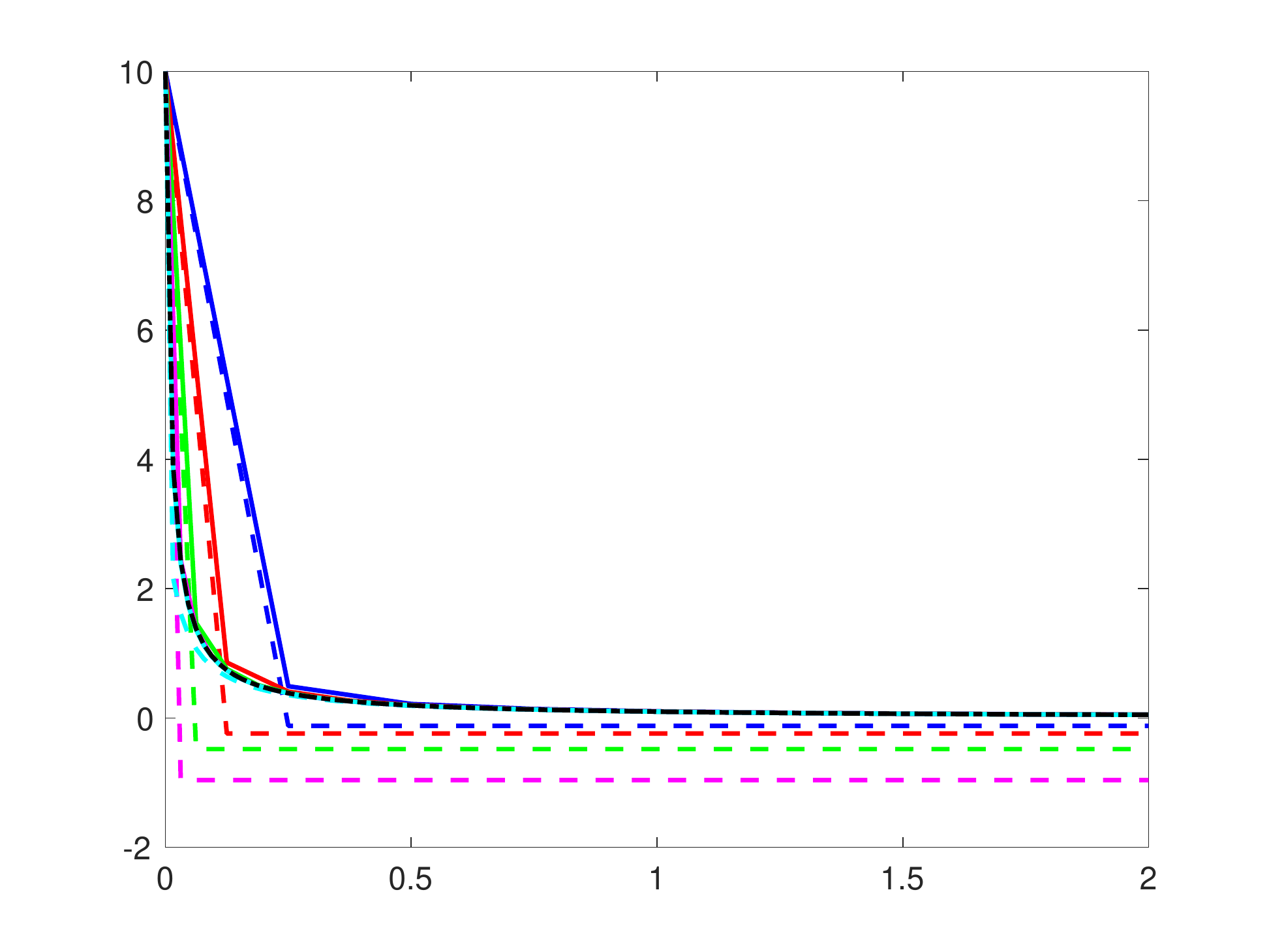}
        \caption{The solution of $u'=-10 u^2$ for the DIRK (dashed lines) and SSP-iMDRK (solid lines)
        compared to the correct solution (dash-dot line) for  $\dt= \frac{1}{n}$ where $n=4,8,16, 32, 64$
        in blue, red, green, magenta, and cyan, respectively.
        We see that if $\dt$ is not small enough the qualitative behavior of the numerical solution using 
        the DIRK methods is poor. However, the SSP-iMDRK methods converge to a solution that
        is qualitatively correct for all values of $\dt$ tested. Left: second order methods. Right: third order methods.
         \label{fig:iRKODE} }
         \end{center}
\end{figure}

We test all three of our methods on the nonlinear scalar  problem
\[ u_t = - 10 u^2,\] with initial condition \rev{$ u(0)= 10,$}
 with $T_{final} = 2$. Here, $G=  -10 u^2$ and $\dot{G} =  200 u^3$. This problem satisfies the forward Euler condition for positivity:
 \[ u^n >0 \; \; \Rightarrow \; \;   u^{n+1} = u^n + \dt G(u^n) = u^n \left( 1- 10 \dt u^n \right) > 0, \; \; \; \; \mbox{for} \; \;  \dt \leq \frac{0.1}{u^n},\]
 and the  backward derivative conditions for positivity: 
\[ u^n >0 \;  \Rightarrow \;   u^{n+1} = u^n - \dt^2 \dot{G}(u^n) = u^n \left( 1- 200 \dt^2 (u^n)^2 \right) > 0, \; \mbox{for} \; \;  \dt^2 \leq \frac{0.005}{ (u^n)^2}.\]
 Note that these restrictions induce a severe time-constraint, especially as $u^n$ is \rev{large}, on an explicit method. However, as
long as these (explicit-type) conditions hold for non-zero $\dt$, we preserve this positivity property {\em   unconditionally}
 for the implicit methods we found above. 
 
 We compare our second and third order methods in the subsection above 
 to diagonally implicit stiffly stable methods in the literature, with  Butcher tableau  \cite{CarpenterRK}
\begin{center}
\[ \begin{array}{c|cc} 
\multicolumn{3}{c}{\mbox{Second order}} \\ 
\multicolumn{3}{c}{\mbox{DIRK}} \\ [6pt] 
0 & 0 & 0 \\  [3pt]
1 & \frac{1}{2} & \frac{1}{2}   \\ [3pt]  \hline  \\ [-1pt]
&   \frac{1}{2} & \frac{1}{2}  \\ 
 \end{array}   \hspace{1in}
  \begin{array}{c|rrrr} 
\multicolumn{5}{c}{\mbox{Third order DIRK}} \\ [6pt] 
0 &  0 & 0 &  0 &  0 \\
 \frac{3}{2} & \frac{3}{4} &    \frac{3}{4} & 0  & 0 \\ [3pt]
\frac{7}{5} &   \frac{447}{675} &   - \frac{357}{675}  &  \frac{855}{675} &  0 \\ [3pt]
1&  \frac{13}{42} &   \frac{84}{42} &   - \frac{125}{42} &  \frac{70}{42} \\ [6pt]  \hline  \\ [-9pt]
&  \frac{13}{42} &   \frac{84}{42} &   - \frac{125}{42} &  \frac{70}{42} \\ 
 \end{array} \]
 \end{center}

As expected, the SSP methods preserve positivity up to a large time-step, while the DIRK methods 
lose positivity for relatively small time-steps.  The second order DIRK method loses positivity 
for $\dt > \frac{1}{50}$ and the third order for $\dt > \frac{1}{75}$.
This loss of positivity has significant consequences
to the convergence of the schemes. We see in  Figure  \ref{fig:iRKODE}  
that the DIRK methods  converge to a solution that is qualitatively \rev{poor}  if the time-step is not small. 
On the other hand, the unconditionally  SSP  methods  converge to a solution that is qualitatively
correct even for much larger time-steps.

 \section{Multi-derivative IMEX methods}    \label{sec:2D-IMEX}
 In this section we consider equations of the form \eqref{IMEXeqn}:
\[ u_t = F(u) + G(u),\] 
where  \rev{the time-step restriction coming from $F$ is of a reasonable size (i.e. $F$ is non-stiff), but
the time-step restriction coming from $G$ is very small (i.e. $G$ is stiff).}
We wish to alleviate this time-step restriction.
When dealing with linear stability, we typically turn to IMEX methods to
 alleviate the time-step restriction coming from $G$. However, when we consider 
 more general norms, semi-norms, or convex functionals, the use of IMEX schemes
{\em does not} result in the removal of the time step restriction caused by the operator $G$,
as shown in \cite{Higueras06,IMEX}. 
 \rev{Now that we have showed that unconditional multi-derivative SSP methods exist under 
 the backward derivative conditions, we wish to leverage this knowledge to develop
SSP IMEX methods that avoid a time-step restriction coming from $G$.
We do this by using an explicit SSP solver for the non-stiff term $F$, coupled with a
purely (or diagonally) implicit solver for the stiff term $G$.}

We assume that the operators $F$ and $G$ preserve some nonlinear stability properties under a convex functional $\| \cdot \|$:
\[ \mbox{\bf Condition 1:} \; \; \; \;  \| u + \dt F(u)\| \leq \| u \|  \; \; \;   \mbox{ for all } \;\; \dt \leq \DtFE, \]
for some $\DtFE>0$, and
\[ \mbox{\bf Condition 2:} \; \; \; \;   \| u + \dt G(u)\| \leq \| u \|  \; \; \;   \mbox{ for all } \;\; \dt \leq k\DtFE \]
for some $k>0$, which may be very small.

The backward derivative condition is natural and relevant in many cases (see Subsection  \ref{sec:models}); 
we assume that  $\dot{G}(u) = G'(u) G(u) $ satisfies:
\[  \mbox{\bf Condition 3:} \; \; \; \;   \| u - \dt^2 \dot{G}(u)\| \leq \| u \|  \; \; \;   \mbox{ for all } \;\;  \dt^2 \leq \dot{k} \; \DtFE^2, \]
(where $\dot{k}>0$ can be of any size). Just as above for the implicit methods, we can devise SSP 
IMEX methods where there is no time-step restriction coming from   $G$ or $\dot{G}$, so that the time-step restriction 
depends only on $F$. 


For problem \eqref{IMEXeqn}, we propose an $s$-stage multi-derivative IMEX method, written in the 
 Shu-Osher formulation, as follows
\begin{subequations} \label{IMEX-RK}
\begin{eqnarray}
&&u^{(i)}  =  r_i u^n + \sum_{j=1}^{i-1} p_{ij} u^{(j)}    + \sum_{j=1}^{i-1} w_{ij} \left( u^{(j)}   + \frac{\dt}{r} F(u^{(j)}  ) \right)  
\\
&& \hspace{0.8in}  + \dt d_{ii} G(u^{(i)}) +  \dt^2 \dot{d}_{ii} \dot{G}(u^{(i)}), \quad  \; \; i=1, . . . , s,   \nonumber \\
&&u^{n+1} =  u^{(s)}.
 \end{eqnarray}
 \end{subequations}
 \rev{The value of $r>0$ in the canonical Shu-Osher formulation gives us the 
 SSP coefficient of the explicit method.
 While at first glance it seems that requiring all the forward Euler steps in the method to have the same time-step
 $\frac{\dt}{r}$ is restrictive, in fact this form does not result in loss of generality, as discussed in \cite{GKS11}.
Note that  the terms $G$ and $\dot{G}$ appear only implicitly, so that there is no SSP restriction
arising from the implicit method.}
 
The intermediate stages can be conveniently written in a matrix form:
\begin{eqnarray} \label{IMEX-RKmatrix}
 U = \mR  \ve u^n + \mP U  +  \mW \left( U + \frac{\Delta t}{r} F(U) \right)  +  \dt \mD G(U) +  \dt^2 \mDdot  \dot{G}(U),
 \end{eqnarray}
 where  $ \mP$, $\mW$, and $\mR = I - \mP - \mW$ are $s \times s$ matrices, $r_i$ are the $i$th row sum of $\mR$,  
 $\mD$ and $\mDdot$ are $s \times s$ diagonal matrices, and $\ve$ is a vector of ones. The numerical solution $u^{n+1} $ is then given by the final element of the vector $U$.

 \subsection{SSP properties of multi-derivative IMEX Runge--Kutta} \label{sec:SSPIMEX}
 
  The Shu-Osher form allows us to easily observe the strong stability preserving properties of the method:
\begin{thm}  \label{thm:SSP} 
Given operators $F$ and $G$ that satisfy Conditions 1, 2, and 3, with 
values $\DtFE>0$, $k>0$, $\dot{k}>0$,
 for  some convex functional $\| \cdot \|$, 
and  if the method given by  \eqref{IMEX-RKmatrix} with $r>0$ satisfies the  \rev{componentwise} conditions
\begin{eqnarray} \label{coefIMEX}
\mR \ve  \geq  0, \; \; \; \;  \mP  \geq  0, \; \; \; \;  \mW  \geq  0, \; \; \;  \mD \geq  0, \; \; \; \; \dot{\mD} \leq  0, 
 \end{eqnarray}
then it preserves the strong stability property
\[  \|u^{n+1} \| \leq \| u^n \|  \]
under the time-step condition \[  \dt \leq r \DtFE. \]
\end{thm}

\begin{proof}

Each stage of the method is 
\begin{eqnarray*} \label{RHS}
u^{(i)}     & = & \left( r_i u^n + \sum_{j=1}^{i-1} p_{ij} u^{(j)}  +
 \sum_{j=1}^{i-1} w_{ij} \left(  u^{(j)} + \frac{\dt}{r} F(u^{(j)}) \right)  \right) \\
& & +    \left( \dt d_{ii} G(u^{(i)})  + \dt^2 \dot{d}_{ii} \dot{G}(u^{(i)}) \right) .
 \end{eqnarray*}
\rev{In particular, the first stage is 
\[ u^{(1)}      =  u^n +   \left( \dt d_{11} G(u^{(i)})  + \dt^2 \dot{d}_{11} \dot{G}(u^{(i)}) \right) .\]
Following the argument in the Proof of Theorem \ref{thm:SSPMDRK}, we easily show that
$ \| u^{(1)} \| \leq \|u^n\|.$}

\rev{Now we assume that for the $i$th stage, we start with the $i-1$ previous stage values, each 
of which satisfy $\| u^{(j)} \| \leq \|u^n\| $.The explicit part of the $i$th stage is defined by}
\[
u_e^{i} =  r_i u^n + \sum_{j=1}^{i-1} p_{ij} u^{(j)}  +
 \sum_{j=1}^{i-1} w_{ij} \left(  u^{(j)} + \frac{\dt}{r} F(u^{(j)}) \right)  .
\]
 \rev{We note that this value depends only on previous stages and the operator $F$. }
Given   the non-negativity of all the coefficients \eqref{coef}  we can show that 
\begin{align*}
\| u_e^{i}  \| & =   \left\| r_i u^n + \sum_{j=1}^{i-1} p_{ij} u^{(j)}  +
 \sum_{j=1}^{i-1} w_{ij} \left(  u^{(j)} + \frac{\dt}{r} F(u^{(j)}) \right) \right\|  \\
 & \leq   r_i  \| u^n \| 
+  \sum_{j=1}^{i-1} p_{ij} \| u^{(j)} \| + 
 \sum_{j=1}^{i-1} w_{ij} \left\| u^{(j)} + \frac{\dt}{r} F(u^{(j)}) \right\| \\
 & \leq  r_i  \| u^n \| 
+  \sum_{j=1}^{i-1} p_{ij} \| u^{(j)} \| + 
 \sum_{j=1}^{i-1} w_{ij} \left\| u^{(j)} \right\|,
 \end{align*}
 for all $\dt \leq r \DtFE$. \rev{Now, recalling  that $\| u^{(j)} \| \leq \|u^n\|$ for $j < i$, we obtain 
$ \left\| u_e^{i}   \right\|  \leq \|u^n\|, $}
 from the condition $R+W+P =I$.\\
 \rev{We now have 
$u^{(i)}    =  u_e^{i}   +   \dt d_{ii}  G(u^{(i)})  +  \dt^2  \dot{d}_{ii}  \dot{G}(u^{(i)}) $
where $\left\| u_e^{i}   \right\|  \leq \|u^n\|$.
 Using Conditions 2 and 3 and the argument in the proof of Theorem \ref{thm:SSPMDRK} above, 
 we can show that} $\| u^{(i)} \| \leq \| u_e^{i}   \|$, whenever $d_{ii} \geq 0$ and $\dot{d}_{ii} \leq 0$,
 and so
$ \| u^{(i)} \| \leq \| u^n\| $ under the  time-step $ \dt \leq r \DtFE. $
\end{proof}

\smallskip

In Subsection \ref{sec:new_methods} we will show that it is indeed possible to find second and third order 
methods that satisfy the requirements in Theorem \ref{thm:SSP}. However, we first present the order conditions
a method of this form must satisfy. 

\subsection{Formulating order conditions} \label{sec:OC-IMEX}
The order conditions for a method \eqref{IMEX-RK}  are generally easier to formulate if the method is written in its Butcher form:
\begin{subequations} \label{IMEXRKButcher}
\begin{eqnarray}
u^{(i)}  & = & u^n  + \dt  \sum_{j=1}^{i-1} \hat{a}_{ij} F(u^{(j)})   + \dt  \sum_{j=1}^{i}  a_{ij} G(u^{(j)}) 
+  \dt^2 \sum_{j=1}^{i}  \dot{a}_{ij} \dot{G}(u^{(j)}), \\
&& \hspace{2in}  i=1, . . . , s,  \nonumber \\
\; \; \;  \; \; \; \; u^{n+1} & = & u^n  + \dt  \sum_{j=1}^{i-1} \hat{b}_{j} F(u^{(j)})   +  \dt \sum_{j=1}^{i}  b_{j} G(u^{(j)}) 
+ \dt^2 \sum_{j=1}^{i}  \dot{b}_{j} \dot{G}(u^{(j)}) .
 \end{eqnarray}
 \end{subequations}
To be consistent with (\ref{IMEX-RK}), we require that $u^{n+1} = u^{(s)}$, so that 
$ \hat{b}_{j} = \hat{a}_{sj}, \;  {b}_{j} = {a}_{sj}, \;    \dot{b}_{j} = \dot{a}_{sj}$.
The  intermediate stages of this method can be written in a matrix form:
\begin{equation}\label{IMEX-RKmatrix1}
U = \ve u^n +   \Delta t \widehat{\mA} F(U)+ \Delta t \mA G(U) +  \Delta t^2  \dot{\mA} \dot{G}(U). 
\end{equation}

The conversion between the two formulations (\ref{IMEX-RKmatrix})   and (\ref{IMEX-RKmatrix1}) is given by:
\begin{eqnarray}
\widehat{\mA}   =  \frac{1}{r} \mR^{-1} \mW, \quad  \mA  =  \mR^{-1} \mD , \quad  \dot{\mA}   =  \mR^{-1} \dot{\mD}.
\end{eqnarray}

The vectors $\widehat{\vb}$,  $\vb$, and $\dot{\vb}$ are given by the last row of $\widehat{\mA}$,  $\mA$, and $\dot{\mA}$, respectively.
The vectors $\vc=\mA \ve$, $\Cdot=\dot{\mA} \ve$, and $\Chat= \widehat{\mA} \ve$ define the time-levels at which the stages are happening;
these values are known as the abscissas.   The order conditions for methods of this form are:

\smallskip

\begin{minipage}{.45\textwidth}
{\renewcommand{\arraystretch}{1.4}
\begin{tabular}{|lc|} \hline 
{For $p \geq 1$} &
$\vb^t \ve=1$\\
& $\bhat^t \ve=1$ \\ \hline 
{For p $\geq$ 2}  &$\vb^t \vc + \bdot^t \ve=\frac{1}{2}$ \\
&$\vb^t\Chat =\frac{1}{2} $ \\
& $\bhat^t \vc =\frac{1}{2} $  \\
& $\bhat^t \Chat =\frac{1}{2}  $\\ \hline
{For $p \geq 3$}  & $ \vb^t\mA \vc + \bdot^t \vc + \vb^t \Cdot = \frac{1}{6}  $ \\
& $ \vb^t\mA\Chat + \bdot^t\Chat = \frac{1}{6}  $ \\
& $ \vb^t\widehat{\mA} \vc =\frac{1}{6} $ \\
& $ \vb^t\widehat{\mA}\Chat = \frac{1}{6} $ \\ \hline
\end{tabular}}
    \end{minipage}
\begin{minipage}{.45\textwidth}
{\renewcommand{\arraystretch}{1.4}
\begin{tabular}{|lc|} \hline
{For $p \geq 3$} 
& $ \bhat^t\mA \vc + \bhat^t \Cdot = \frac{1}{6} $ \\
(continued) & $ \bhat^t \mA \Chat  = \frac{1}{6} $ \\
& $ \bhat^t\widehat{\mA} \vc =\frac{1}{6} $ \\
& $ \bhat^t\widehat{\mA}\Chat = \frac{1}{6} $ \\ 
& $\vb^t(\vc\cdot \vc) + 2\bdot^t \vc = \frac{1}{3} $ \\
& $\vb^t( \vc\cdot \Chat) + \bdot^t \Chat = \frac{1}{3} $ \\
& $\vb^t( \Chat \cdot \Chat)  = \frac{1}{3} $ \\
& $\bhat^t(\vc\cdot \vc) = \frac{1}{3} $ \\
& $\bhat^t( \vc\cdot \Chat)  = \frac{1}{3} $ \\
& $\bhat^t( \Chat \cdot \Chat)  = \frac{1}{3}$ \\ \hline
\end{tabular}
}
\end{minipage}

\smallskip


\subsection{New SSP IMEX multi-derivative Runge--Kutta methods}\label{sec:new_methods}
Given functions $F$ and $G$ that satisfy Conditions 1-3, 
these IMEX methods have an explicit part that is SSP for a time-step that depends only on $F$,
and an implicit part that is unconditionally SSP. We will later show that these methods are 
positivity preserving and also asymptotic preserving for the problems described in 
Subsection \ref{sec:models}.

\subsubsection{Second order method}
We begin with a method that has Shu-Osher coefficients 
\[ \mW = \left[ \begin{array}{lll}
0 & 0 & 0 \\
1 & 0 & 0 \\
0 & 1/2 & 0 \\
\end{array} \right] , \; \; \; \;
 \mP =  \left[ \begin{array}{lll}
0 & 0 & 0 \\
0 & 0 & 0 \\
1/2 & 0 & 0 \\
\end{array} \right] , \; \; \; 
\mR\ve=  \left[ \begin{array}{l}
    1     \\
    0 \\
   0 \\
\end{array} \right] ,\]
and
\[ diag(\mD) = 
\left[ \begin{array}{c}
   \frac{1}{2}   \\
    0 \\
   \frac{1}{2}  \\
\end{array} \right] ,\;\;\;
diag(\dot{\mD}) =  - 
\left[ \begin{array}{c}
   0  \\
    \frac{1}{2}  \\
   0  \\
\end{array} \right],
\]
with $r=1$.

 In Butcher form, these become 
\[ \widehat{\mA} =  \left[ \begin{array}{lll}
0 & 0 &  0\\
1 & 0 & 0 \\
1/2 & 1/2 &  0\\ 
  \end{array} \right], 
\; \; \; 
\mA= \left[ \begin{array}{lll}
1/2  & 0 &  0\\
 1/2  & 0 &  0 \\
  1/2 & 0 & 1/2 \\
  \end{array} \right],  \; \; \; 
\dot{\mA} =   \left[ \begin{array}{lll}
0 & 0 &  0\\
0  & -1/2 &  0 \\
0 & -1/4 & 0 \\
  \end{array} \right].
   \] 
 The benefit of this method over the one in \cite{HSZ18} is that the positivity preserving  coefficient $r=1$ for this method is larger than
 the  positivity preserving  coefficient 
 $r= 0.8125$ in the method given in Subsection 2.6.2 of \cite{HSZ18}. 
 The two methods each require the implicit solution of three stages.

\subsubsection{Third order method}
We found a third order method of this form, as well.
This method has  $r=0.904402174130635$
with coeffiicients:

\[ diag(\mD) = \left( \begin{array}{c}
0 \\
   2\\\
   0.388820513661584\\
   0.083529464436389\\
   1.793313488277995\\
0\\
  \end{array} \right), \; \; \; \; 
  diag(\dot{\mD}) = - \left( \begin{array}{c}
0.871358934880525\\
  0.856842702601821\\
  0\\
  0\\
  2\\
  0.205134529930013\\
  \end{array} \right).
  \]
  Note that $d_{ii} + \left| \dot{d}_{ii} \right| >0 $ for each stage $i$.
  
  \[ \mW = \left( \begin{array}{cccccc}
                   0                  &      0         &   \hspace{-0.05in}        0        &           0        &            0 & 0 \\
   0.058453072749259   &       0         &  \hspace{-0.05in}         0        &           0        &            0 & 0 \\
   0.764266518291495   &       0         &  \hspace{-0.05in}         0        &           0        &            0 & 0 \\
      0 			     & 	      0 & \hspace{-0.05in}    0.292520982667463           &        0         &          0 & 0 \\
   0.173788618990251   &      0 &\hspace{-0.05in}  0 &   0.281050180194829       &            0 & 0 \\
   0.016811671845949   &      0 & \hspace{-0.05in} 0 &    0.448630511341543   & 0 & 0 \\
     \end{array} \right), \]
     
\[ \mP = \left( \begin{array}{cccccc}     
   0            &       0         &          0        &           0        &            0 & 0 \\
   0.253395246357353             &       0         &          0        &           0        &            0 & 0 \\
   0   &0.235733481708505       &          0        &           0        &            0 & 0 \\
   0   &0.123961833526104   &          0        &           0        &            0 & 0 \\
   0.409037644509411   & 0.136123556305509  &          0        &           0        &            0 & 0 \\
   0.203353399602184   & 0 & 0 & 0 &   0.331204417210324 & 0 \\
     \end{array} \right), \]
   
   \[\mR\ve = \left( \begin{array}{l}  
     1 \\
   0.688151680893388 \\
                   0 \\
   0.583517183806433 \\
                   0 \\
                   0 \\
 \end{array} \right).
  \]

We have the Butcher form coefficient matrices:
{\scriptsize
\[ \mAh = \left( \begin{array}{cccccc}  
       0      &             0       & 0   &        0           &        0        &           0 \\
   0.064631725156397   & 0   &        0            &       0         &          0     &              0 \\
   0.860287477078593   & 0   &        0            &       0         &          0      &              0 \\
   0.259664005325885   & 0   & 0.323441264334256  &                 0         &       0    & 0 \\
   0.273935075266107   & 0   & 0.090903225623586 &  0.310757966128278        &   0          &         0 \\
   0.225810414773773   & 0   & 0.175213169672431  & 0.598976415553796        & 0            &       0 \\
   \end{array} \right),
  \] 

\[ \mA = \left( \begin{array}{cccccc}  
  0&                   0                 &  	0            &       		0        &           		0           &         0 \\
  0 &  2   &                0          &                        0             &      		0        &            0 \\
  0 &  0.471466963417009   & 0.388820513661584&                   0           &        		0         &          0\\
  0 &  0.385837646486197   & 0.113738158737554&   0.083529464436389           &       0           &        0 \\
  0 &  0.380686852681912   & 0.031966130008218&   0.023475971031425   & 1.793313488277995  & 0 \\
  0 &  0.299183707820065   & 0.061613731773316&   0.045249211646092   & 0.593953348760527  &  0 \\
   \end{array} \right),
  \]}
and
{\scriptsize
\[ \mAdot = -  \left( \begin{array}{cccccc}  
  0.871358934880525&                   0                 & \hspace{-.05in} 0   & \hspace{-.05in}0   &  0    & 0 \\
  0.271731819181020&  0.856842702601821  &\hspace{-.05in}0  &\hspace{-.05in}0    & 0     & 0 \\
  0.730006747169852&  0.201986513560852  &\hspace{-.05in}0  &\hspace{-.05in}0    & 0     & 0 \\
  0.247226665569066&  0.165301085890380  &\hspace{-.05in}0  &\hspace{-.05in}0    & 0     & 0 \\
  0.614323072678900&  0.163094375848475  &\hspace{-.05in}0  &\hspace{-.05in}0    & 2    & 0 \\
  0.506222742811925&  0.128176688391489  &\hspace{-.05in}0  &\hspace{-.05in}0    & 0.662408834420649 &  0.205134529930013 \\
    \end{array} \right).
  \] 
 }

To achieve a third order method we required six stages.
However, this allowed us to design a third order method that is SSP with a time-step restriction that does not
depend on $G$.

\subsection{Applications}
The new SSP multi-derivative IMEX methods developed in Subsection \ref{sec:new_methods} are of particular use for 
a number of models we describe in Subsection \ref{sec:models}. These are all problems that lead to  ODE systems of the form:
\begin{equation} \label{uu}
\frac{\rd{}u}{\rd{}t} = T(u)+ \frac{1}{\varepsilon} Q(u),
\end{equation}
where the solution $u(t)\in \mathbb{R}^N$, and the operators  $T$, $Q$: $\mathbb{R}^N\rightarrow \mathbb{R}^N$, $N\geq 2$. 
The parameter $0<\varepsilon \leq O(1)$ indicates the regime of the problem: $\varepsilon=O(1)$ corresponds to the non-stiff regime; 
$\varepsilon \ll 1$ to the stiff regime. 
For such systems, we require a high order  time discretization that  preserves the physical properties at the discrete level,
in particular  positivity and the asymptotic limit. 

\smallskip

\noindent{\bf Positivity:}
Problems of the form \eqref{uu} that are of interest to us have positive solutions. It is preferable that the numerical solution will preserve this positivity property, for a time-step not dependent on $\varepsilon$. It should be pointed out that positivity is an important property when solving kinetic equations. For example, the Bhatnagar-Gross-Krook (BGK) model (see equation (\ref{BGK}) below) requires the macroscopic quantities to be positive, and even small negative values of the solution $f$ may cause some macroscopic quantities, especially the temperature, to fail to be well-defined. In such cases, the requirement that the numerical solution remains positive for time-steps independent of $\varepsilon$ is critical to the success of the simulation. Strong stability preserving methods are also positivity preserving, so the multi-derivative IMEX methods given in 
Subsection \ref{sec:new_methods} will preserve these properties, with a time-step independent of $\varepsilon$.

\smallskip

\noindent{\bf Asymptotic limit:}
Very often the operator $Q$ satisfies the following properties: $Q$ is ``conservative" in the sense that there exists a linear operator $\cR$: $\mathbb{R}^N\rightarrow \mathbb{R}^n$, $n<N$, s.t. $\cR Q(u)=0$, $\forall \ u$; $Q$ is dissipative and has a unique local equilibrium of the form $E(\cR u)$, where $E$: $\mathbb{R}^n\rightarrow \mathbb{R}^N$ is some operator. Using these properties, applying $\cR$ to (\ref{uu}) yields
\begin{equation} \label{uu1}
\frac{\rd{}\omega}{\rd{}t} = \cR T(u), \quad \omega :=\cR u,
\end{equation}
which is not a closed system. However, if $\varepsilon \rightarrow 0$, (\ref{uu}) implies $Q(u)\rightarrow 0$, hence $u\rightarrow E(\omega)$. Substituting this $u$ into (\ref{uu1}) gives a closed (reduced) system:
\begin{equation} \label{ww}
\frac{\rd{}\omega}{\rd{}t} = \cR T(E(\omega)).
\end{equation}
The above simple analysis reveals that when $\varepsilon \rightarrow 0$, (\ref{uu}) is not only stiff but also possesses a {\em non-trivial asymptotic limit}. (Recall that $n<N$ and note that the original variable $u$ is in $\mathbb{R}^N$ while the reduced variable $\omega$ is in $ \mathbb{R}^n$.)

Systems of the form (\ref{uu}) arise (after the method of lines discretization of a PDE) from many physical problems in multi-scale modeling. A prominent example is the Boltzmann equation in kinetic theory \cite{Cercignani}:
\begin{equation}  \label{Boltz}
\partial_tf +v\cdot \nabla_x f= \frac{1}{\varepsilon} Q(f), \quad x, v\in \mathbb{R}^d, 
\end{equation}
where $f=f(t,x,v)\geq 0$ is the probability density function of time $t$, position $x$, and velocity $v$.
The term $v\cdot \nabla_xf$ describes the particle transport, and $Q(f)$ describes the collisions 
between particles, which is a complicated nonlinear integral operator. 
The dimensionless parameter $\varepsilon$, called the Knudsen number,
 is defined as the ratio of the mean free path and characteristic length scale. 
 When $\varepsilon = O(1)$, the transport and collision balance so the system is in the fully kinetic regime. 
 When $\varepsilon \ll 1$, the collision effect dominates, i.e., collisions happen so frequently that the overall 
 system is close to the local equilibrium or fluid regime. In this case, one can derive the limiting fluid equations (the compressible Euler equations) as $\varepsilon \rightarrow 0$ from (\ref{Boltz}). The process is similar to the abstract model reduction procedure described above for (\ref{uu}). 

We require a time-stepping method that preserves the asymptotic limit of the equation. 
That is, for a fixed $\Delta t$, when $\varepsilon\rightarrow 0$, the scheme for (\ref{uu}) automatically reduces to a high order time discretization for the limiting system (\ref{ww}). 
A numerical scheme with this property is called asymptotic preserving (AP) as initially coined in \cite{Jin99}. 
To insure the AP property, the time step $\Delta t$ should not be limited by the small parameter $\varepsilon$. 
This necessitates some implicit treatment of the stiff collision term $\frac{1}{\varepsilon} Q(u)$. 
The need for the AP property further motivates the  use of implicit-explicit (IMEX) methods. \rev{There is an extensive literature on development of IMEX schemes that possess the AP property, see, for instance, \cite{PR05, DP13, BPR13, DP17} for the application to hyperbolic and kinetic equations.}

The need for a high order numerical integrator that is both asymptotic preserving and positivity preserving motivated
the work in this paper. We will show that the second and third order methods we presented above are 
asymptotic preserving high order time discretization methods that preserve the positivity of the solution for arbitrary 
$\varepsilon$. Previously, designing a time-stepping scheme with both positivity and AP property has proven difficult. 
First order  IMEX schemes with these properties exist, but methods above first order may violate  
positivity unless the time-step is restricted by $\varepsilon$ \cite{Higueras06, HR06}. 

Second order IMEX schemes that preserve the AP property and positivity for arbitrary $ \varepsilon$ have been 
previously found,  by incorporating a derivative correction term  at the final stage of each time-step.
 Such an approach was successfully considered in  \cite{HS17, HSZ18} 
 (note that the method in \cite{HS17} only works for a special relaxation system and
  can preserve the positivity of one component of the solution vector, while \cite{HSZ18} 
  works for a general class of equations and the scope is similar to what we consider in this work);  
  however, this strategy failed to find methods of order three. By formulating IMEX multi-derivative  Runge--Kutta methods that allow the 
  use of  $\dot{Q}$ at every stage,  we are able to obtain a third order IMEX method that is AP and positivity preserving independent of $\varepsilon$. 
 Furthermore, the second order method  improves upon the  previously presented method in \cite{HSZ18}, 
 in the sense that we obtain a $23\%$  larger allowable time-step. 
%

We present a a summary of the model equations  and their properties in Subsection \ref{sec:models}.  
In Subsection \ref{APIMEX} we prove the positivity and asymptotic preserving properties of the multi-derivative
IMEX Runge--Kutta methods. Finally, in Subsection  \ref{sec:numerical} we demonstrate the numerical performance 
of these methods on sample problems. 

\subsubsection{A summary of the models and properties} \label{sec:models}

We assume that the operators $T$ and $Q$ in (\ref{uu}) satisfy the following properties:

\begin{description}[leftmargin=2em,style=nextline]

\vspace{0.1in}
 \item[Property 1] {\em The operator $T$ is conditionally positivity preserving under a forward Euler step: }
 \begin{equation}
u >0 \Longrightarrow u+\Delta tT(u) >0, \  \forall  \  0\leq \Delta t \leq \Delta t_{\text{FE}},
 \end{equation}
 {\em for some  time step $\Delta t_{\text{FE}}> 0$.}
 
 \vspace{0.1in}
  \item[Property 2]  {\em The operator $Q$ is unconditionally positivity preserving under a backward Euler step: }
  \begin{equation}
u > 0, \ v=u+\Delta tQ(v)  \Longrightarrow v> 0, \  \forall  \ \Delta t\geq 0.
  \end{equation}
    \end{description}
  We observe that the first two properties essentially concern the positivity preserving property of the operators $T$ and $Q$ in equation (\ref{uu}). 
  
  \rev{
  \begin{rmk}
  Property 2 plays a similar role to that of Condition 2 in Section \ref{sec:2D-IMEX}. 
 Condition 2 is a forward Euler condition \eqref{FEcond}, which we then use to show that the backward Euler method
 unconditionally preserves this strong stability property. Property 2 states that backward Euler preserves 
 positivity unconditionally.  This is necessary because positivity may be preserved under the forward Euler condition but
 be violated (for certain $\dt$) for the backward Euler method.
  \end{rmk}}
  
  \vspace{0.1in}
  \begin{description}[leftmargin=2em,style=nextline]
    \item[Property 3] {\em  Conservation of $Q$: there exists a linear operator $\cR: \mathbb{R}^N\rightarrow \mathbb{R}^n$, $n< N$, s.t.}
  \begin{equation}
\cR Q(u)=0, \ \forall \ u.
  \end{equation}
  
  \vspace{0.1in}
 \item[Property 4]  {\em Equilibrium of $Q$: there exists an (possibly nonlinear) operator $E: \mathbb{R}^n\rightarrow \mathbb{R}^N$, s.t.}
  \begin{equation}
Q(u)=0 \Longrightarrow u=E(\cR u).
 \end{equation}
{\em Moreover, $E$ satisfies $\cR E(\cR u)=\cR u$, $\forall \ u$.}
\end{description}
 Note that Properties 3 and 4 together imply that (\ref{uu}) has a limiting system (\ref{ww}). Properties 1--4 are satisfied by a large class of kinetic equations of the form (\ref{Boltz}), where the collision operator $Q$ can be the full Boltzmann collision operator (an integral type operator), the kinetic Fokker-Planck operator (a diffusion type operator), the BGK operator (a relaxation type operator), or its generalized version such as the ES-BGK operator. For more details about these operators, we refer the readers to \cite{HS19}.

   \vspace{0.1in}
    \begin{description}[leftmargin=2em,style=nextline]
  \item[Property 5] {\em  The Fr\'{e}chet derivative of $Q$ satisfies }
   \begin{equation}
  \dot{Q}(u):=Q'(u)Q(u)=-C_{\cR u}Q(u),
  \end{equation}
  {\em  where $C_{\cR u}$ is some positive function depending only on $\cR u$. }
  {\em  The Fr\'{e}chet derivative of $Q$ at $u$ is defined by}
  \begin{equation}
Q'(u)v = \lim_{\delta\rightarrow 0}\frac{Q(u+\delta v)-Q(u)}{\delta}.
 \end{equation}
\end{description}
Property 5 means the operator $Q$ is dissipative in some sense. This property  is not generic but it is satisfied by quite a few kinetic models including the BGK operator and the Broadwell model. Some stiff ODE systems and hyperbolic relaxation systems also satisfy this property, though for these problems positivity is usually not a big concern compared to the kinetic equations. Since our proposed multi-derivative methods highly depend on Property 5, we list below a few examples.

\vspace{0.1in}
\begin{itemize}
\item {\bf An ODE model:}
\begin{equation} \label{vdp}
\left\{
\begin{split}
u_1'&=u_2,\\
u_2'&=\frac{1}{\varepsilon}f(u_1)\left(g(u_1)-u_2 \right),
\end{split}
\right.
\end{equation}
where $f$ and $g$ are some functions of $u_1$, and $f(u_1)> 0$. Define 
\begin{equation}
u=(u_1,u_2)^T, \quad T(u)=(u_2,0)^T, \quad Q(u)=\left(0,f(u_1)\left(g(u_1)-u_2 \right)\right)^T,
\end{equation}
then (\ref{vdp}) falls into the general form (\ref{uu}). It is easy to see that (\ref{vdp}) has a limit as $\varepsilon\rightarrow 0$:
\begin{equation}
u_1'=g(u_1).
\end{equation}
Indeed, one can just take \rev{ $\cR u = u_1$} and $E(\cR u)=E(u_1):=\left(u_1,g(u_1)\right)^T$. It can also be verified by direct calculation that 
\begin{equation}
\dot{Q}(u)=-f(u_1)Q(u).
\end{equation}

\item 
{\bf A PDE model: the hyperbolic relaxation system \cite{CLL94}:}
\begin{equation} \label{hr}
\left\{
\begin{split}
&\partial_tu_1+\partial_xu_2=0,\\
&\partial_tu_2+\partial_x u_1=\frac{1}{\varepsilon}\left(F(u_1)-u_2\right),
\end{split}
\right.
\end{equation}
where $F$ is some function of $u_1$.  Equation (\ref{hr}) again has the form of (\ref{uu}) if we define $u=(u_1,u_2)^T$, 
$T(u)=-(\partial_xu_2,\partial_xu_1)^T$, $Q(u)=(0,F(u_1)-u_2)^T$. 
Note that we abused the notation a bit: $u$, $T$ and $Q$ should be defined for the system after spatial discretization. 
It is easy to see that (\ref{hr}) has a limit as $\varepsilon\rightarrow 0$:
\begin{equation}
\partial_t u_1+\partial_xF(u_1)=0.
\end{equation}
Indeed, one can just take \rev{ $\cR u = u_1$} and $E(\cR u)=E(u_1):=\left(u_1,F(u_1)\right)^T$. Similarly to the previous model, it can be verified that 
\begin{equation}
\dot{Q}(u)=-Q(u).
\end{equation}

\item {\bf The Broadwell model \cite{Broadwell64}:}
The Broadwell model is a simple discrete velocity kinetic model:
\begin{equation} \label{broadwell}
\left\{
\begin{split}
&\partial_t f_+ + \partial_x f_+  = \frac{1}{\varepsilon}(f_0^2-f_+f_-), \\
&\partial_t f_0 = - \frac{1}{\varepsilon}(f_0^2-f_+f_-), \\
&\partial_t f_- - \partial_x f_-  = \frac{1}{\varepsilon}(f_0^2-f_+f_-), \\
\end{split}
\right.
\end{equation}
where $f_+=f_+(t,x)$, $f_0=f_0(t,x)$, and $f_-=f_-(t,x)$ denote the densities of particles with speed 1, 0, and $-1$, respectively. Define $f = (f_+,f_0,f_-)^T$, $T(f) = (-\partial_x f_+,0,\partial_x f_-)^T$, and $Q(f) = (f_0^2-f_+f_-,-(f_0^2-f_+f_-),f_0^2-f_+f_-)^T$ (again these should be defined for the system after spatial discretization). Then (\ref{broadwell}) falls into the general form (\ref{uu}). To see its limit as $\varepsilon \rightarrow 0$, we rewrite (\ref{broadwell}) using moment variables:
\begin{equation} \label{broadmoments}
\left\{
\begin{split}
\partial_t \rho + \partial_x m & = 0,\\
\partial_t m + \partial_x z & = 0,\\
\partial_t z + \partial_x m & = \frac{1}{2\varepsilon}(\rho^2 + m^2 - 2\rho z),\\
\end{split}
\right.
\end{equation}
where $\rho:=f_++2f_0+f_-$, $m := f_+-f_-$, and $z:= f_++f_-$. From (\ref{broadmoments}), it is clear that when $\varepsilon\rightarrow 0$, $z\rightarrow \frac{\rho^2+m^2}{2\rho}$. This, when substituted into the first two equations, yields a closed hyperbolic system:
\begin{equation} \label{broadlimit}
\left\{
\begin{split}
&\partial_t \rho + \partial_x m  = 0,\\
&\partial_t m + \partial_x \left (\frac{\rho^2+m^2}{2\rho} \right)  = 0.\\
\end{split}
\right.
\end{equation}
Indeed, the operators $\cR$ and $E$ in Properties 3--4 can be taken as
\begin{eqnarray} 
&\rev{\cR f=(\rho,m)^T,}\\
&E(\cR f)=E((\rho,m)^T):=\left(\frac{(\rho+m)^2}{4\rho},\frac{\rho^2-m^2}{4\rho},\frac{(\rho-m)^2}{4\rho}\right)^T. \nonumber
\end{eqnarray}
Furthermore, it can be verified that
\begin{equation}
\dot{Q}(f) = -\rho Q(f).
\end{equation}
\item {\bf The Bhatnagar-Gross-Krook (BGK) model \cite{BGK54}:}
The BGK model is a widely used kinetic model introduced to mimic the full Boltzmann equation:
\begin{equation} \label{BGK}
\partial_tf +v\cdot \nabla_x f=\frac{1}{\varepsilon} (M-f), \quad x,v \in \mathbb{R}^d,
\end{equation}
where $f=f(t,x,v)$ is the probability density function and $M$ is the so-called Maxwellian given by
\begin{equation}
M(t,x,v)=\frac{\rho(t,x)}{(2\pi T(t,x))^{d/2}}\exp \left(-\frac{|v-u(t,x)|^2}{2T(t,x)}\right),
\end{equation}
where the density $\rho$, bulk velocity $u$ and temperature $T$ are given by the moments of $f$:
\begin{equation}
\rho=\int_{\mathbb{R}^d} f \rd{v}, \quad \rho u=\int_{\mathbb{R}^d}fv \rd{v}, \quad  \frac{1}{2}\rho dT=\frac{1}{2}\int_{\mathbb{R}^d}f|v-u|^2\rd{v}. 
\end{equation}
To see its asymptotic limit, we multiply (\ref{BGK}) by $(1,v,|v|^2/2)^T$ and integrate w.r.t.~$v$ to obtain
\begin{align}  \label{lcl}
\left\{
\begin{array}{l}
\displaystyle\partial_t \rho+\nabla_x\cdot \int_{\mathbb{R}^d} vf \rd{v}=0,\\[8pt]
\displaystyle \partial_t (\rho u) +\nabla_x\cdot  \int_{\mathbb{R}^d} v\otimes vf\rd{v}=0,\\[8pt]
\displaystyle \partial_t \mathcal{E} +\nabla_x\cdot \int_{\mathbb{R}^d} \frac{1}{2}v|v|^2f\rd{v}=0,
\end{array}\right.
\end{align}
where $\mathcal{E}=\frac{1}{2}\rho u^2 +\frac{1}{2}\rho dT$ is the total energy. This system is not closed. However, if $\varepsilon \rightarrow 0$, (\ref{BGK}) implies $f\rightarrow M$. Substituting this $f$ into (\ref{lcl}), we can get a closed system
\begin{align}  \label{euler}
\left\{
\begin{array}{l}
\partial_t \rho+\nabla_x\cdot (\rho u)=0,\\[8pt]
\partial_t (\rho u)+\nabla_x\cdot (\rho u\otimes u +pI)=0,\\[8pt]
\partial_t \mathcal{E}+\nabla_x\cdot ((\mathcal{E}+p)u)=0,
\end{array}\right.
\end{align}
where $I$ is the identity matrix and $p=\rho T$ is the pressure. 
Equation (\ref{euler}) is nothing but the compressible Euler equations. 
To write the BGK model into the form (\ref{uu}), we define $T(f)=-v\cdot \nabla_x f$ 
and $Q(f)=M-f$ (these should be defined for (\ref{BGK}) after spatial and velocity discretization). 
Moreover, the operators $\cR$ and $E$ are given by
\begin{eqnarray}
& \rev{ \cR f=\int_{\mathbb{R}^d} f(1,v,|v|^2/2)^T\,\rd{v}=(\rho,\rho u,\mathcal{E})^T,}\\
&E(\cR f)=E((\rho, \rho u, \mathcal{E})^T)=M.
\end{eqnarray}
Furthermore, it can be verified that 
\begin{equation}
\dot{Q}(f) = -Q(f).
\end{equation}
\end{itemize}

To summarize, we have introduced four different models (including both ODE and PDEs) which all satisfy Properties 3--5. For the Broadwell model and BGK model, one can check that they also satisfy the positivity-preserving Properties 1--2 provided a positivity preserving spatial discretization is used for the transport/convection term, see \cite{HSZ18} for more details.

\subsubsection{Properties of the numerical scheme} \label{APIMEX}
A \rev{multi-derivative IMEX} method that is SSP as shown in Subsection \ref{sec:SSPIMEX} 
will also be positivity preserving. 
This is because the SSP property holds for any 
convex functional, and positivity is preserved under a convex functional. 
In Proposition \ref{prop:AP_positive} we show this explicitly, 
and we also prove that under a mild additional condition satisfied by the methods in 
Subsection \ref{sec:new_methods}, the asymptotic preserving property is satisfied as well.

\begin{proposition} \label{prop:AP_positive}
Assume that the problem (\ref{uu}) satisfies the Properties 1--5  listed in Subsection~\ref{sec:models}. 
Then the scheme (\ref{IMEX-RK}) that satisfies the inequalities (element-wise)
\begin{eqnarray}
 \mR  \ve \geq  0,  \quad  \mP  \geq  0, \quad \mW  \geq  0, \quad \mD \geq  0, \quad \dot{\mD} \leq  0, 
 \end{eqnarray}
will preserve the positivity of the solution for all $\dt \leq r \DtFE$. 
\rev{Furthermore, if we require that at least one of $Q$ or $\dot{Q}$ appear at every stage,
 i.e. the strict inequality
\begin{eqnarray}
d_{ii} + \left|  \dot{d}_{ii} \right|  >  0, \quad \text{for all } \ i=1,\dots,s,
  \end{eqnarray}
is also satisfied,} then  the scheme is AP, \rev{i.e.} when $\Delta t$ is fixed, as $\varepsilon \rightarrow 0$, (\ref{IMEX-RK}) automatically 
reduces to an explicit Runge-Kutta scheme, with the same order as the original scheme, applied to the limiting system (\ref{ww}).
\end{proposition}

\begin{proof}
We consider each stage of  \eqref{IMEX-RK}, 
\begin{eqnarray} \label{scheme1}
u^{(i)} & = & r_i u^n + \sum_{j=1}^{i-1} p_{ij} u^{(j)}    + \sum_{j=1}^{i-1} w_{ij} \left( u^{(j)}   + \frac{\dt}{r} T(u^{(j)}  ) \right)   \\
&& + \; \;  \frac{\dt}{\varepsilon} d_{ii} Q(u^{(i)}) + \frac{\dt^2}{\varepsilon^2} \dot{d}_{ii} \dot{Q}(u^{(i)})  \nonumber \\
 & = & r_i u^n + \sum_{j=1}^{i-1} p_{ij} u^{(j)}    + \sum_{j=1}^{i-1} w_{ij} \left( u^{(j)}   + \frac{\dt}{r} T(u^{(j)}  ) \right)   \nonumber  \\
 && + \; \;  \left(\frac{\dt}{\varepsilon} d_{ii} - \frac{\dt^2}{\varepsilon^2} \dot{d}_{ii}C_{\cR u^{(i)}} \right)Q(u^{(i)}), \nonumber 
\end{eqnarray}
where we applied Property 5 to the last term $\dot{Q}(u^{(i)})$.
  
At the first stage, we have
  \[ u^{(1)} =  u^n + \left(\frac{\dt}{\varepsilon} d_{11} - \frac{\dt^2}{\varepsilon^2} \dot{d}_{11}C_{\cR u^{(1)}} \right)Q(u^{(1)}) .\]
  Given  a positive $u^n$, and since $d_{11}\geq 0$, $\dot{d}_{11}\leq 0$, and $C_{\cR u^{(1)}}> 0$, 
using Property 2 we obtain $u^{(1)}> 0$. 

Now, given a positive $u^n$ and positive stages $u^{(j)}$ for $j<i$, Property 1 gives us the positivity of the explicit terms 
\[\left( u^{(j)}   + \frac{\dt}{r} T(u^{(j)}  ) \right) > 0, \; \; \; \; \mbox{for all }  \;  \frac{\dt}{r} \leq  \DtFE.\]
Consequently, the non-negativity of $r_i$, $p_{ij}$, and $w_{ij}$, together with the fact that $r_i + \sum_{j=1}^{i-1} (p_{ij}+w_{ij}) =1$
ensures the positivity of the explicit terms in $u^{(i)}$
  \[  r_i u^n + \sum_{j=1}^{i-1} p_{ij} u^{(j)} + \sum_{j=1}^{i-1} w_{ij} \left( u^{(j)}   + \frac{\dt}{r} T(u^{(j)}  ) \right) > 0.\]
Finally, since $d_{ii}\geq 0$, $\dot{d}_{ii}\leq 0$, and $C_{\cR u^{(i)}}> 0$, 
Property 2 assures that $u^{(i)}> 0$. 

To see the AP property, we apply $\cR$ to (\ref{scheme1}) to obtain (define $\omega^n=\cR u^n$, $\omega^{(i)}=\cR u^{(i)}$) 
\begin{eqnarray} \label{scheme2}
\omega^{(i)}  & = & r_i \omega^n + \sum_{j=1}^{i-1} p_{ij} \omega^{(j)} + \sum_{j=1}^{i-1} w_{ij} \left( \omega^{(j)} + \frac{\dt}{r} \cR T(u^{(j)}) \right),
  \end{eqnarray}
where the collision terms are gone due to Property 3. On the other hand, when $\Delta t$ 
is fixed and $\varepsilon \rightarrow 0$, since $d_{ii}+|\dot{d}_{ii}|>0$ and $C_{\cR u^{(i)}}> 0$,
we have from (\ref{scheme1}) that $Q(u^{(i)})\rightarrow 0$, hence $u^{(i)}\rightarrow E(\omega^{(i)})$ 
by Property 4. Note that this holds for every $i=1,\dots,s$. Replacing $u^{(j)}$ by $E(\omega^{(j)})$ in (\ref{scheme2}) yields
 \begin{eqnarray*}
\omega^{(i)}  & = & r_i \omega^n + \sum_{j=1}^{i-1} p_{ij} \omega^{(j)} + \sum_{j=1}^{i-1} w_{ij} \left( \omega^{(j)} + \frac{\dt}{r} \cR T(E(\omega^{(j)})) \right), \quad i=1,\dots,s;
  \end{eqnarray*}
together with $\omega^{n+1}=\omega^{(s)}$, this is a high order explicit Runge-Kutta scheme applied to the limiting system (\ref{ww}). 
In fact, it is the explicit part of (\ref{IMEX-RK}) applied to (\ref{ww}).
\end{proof}

\rev{\begin{rmk}
Following the classification of various IMEX Runge-Kutta schemes in \cite{BPR13}, the multi-derivative IMEX schemes introduced in this paper are both type A and GSA. In other words, since $d_{ii} + \left|  \dot{d}_{ii} \right|  >  0$ for all $i$, we are solving an implicit collision step at every stage of the scheme, hence any initial condition is allowed to guarantee the AP property.
\end{rmk}}

\begin{rmk}
In the case of the Broadwell model and BGK equation, Theorem \ref{thm:SSP} can be used to prove the discrete entropy decay property of the numerical method. 
Taking the following 1D BGK equation as an example,
\begin{equation}
\partial_t f+v\partial_x f=\frac{1}{\varepsilon}(M-f).
\end{equation}
We set $G$ to be the BGK operator and $F$ be the transport operator discretized by the first order upwind method ($k$ is the spatial index):
\begin{equation}
(v\partial_xf)_k=\frac{v+|v|}{2}\frac{f_k-f_{k-1}}{\Delta x}+\frac{v-|v|}{2}\frac{f_{k+1}-f_k}{\Delta x},
\end{equation}
together with the periodic or compactly supported boundary condition. The convex functional $\|\cdot \|$ is taken as the discrete entropy 
\begin{equation}
S[f]=\Delta x\sum_k \int f_k\log f_k \rd{v}.
\end{equation}
Then it can be verified that $F$ and $G$ satisfy the Conditions 1--3 (for more details see \cite{HSZ18}). Therefore, the numerical solution obtained by method (\ref{IMEX-RKmatrix}) satisfies 
\begin{equation}
S[f^{n+1}]\leq S[f^n],
\end{equation}
under the conditions listed in Theorem~\ref{thm:SSP}.
\end{rmk}

\subsection{Numerical results} \label{sec:numerical}

In this subsection, we verify the accuracy of the proposed second and third order methods in Subsection \ref{sec:new_methods} on the ODE model, the Broadwell model, and the BGK equation. We will see that the methods exhibit the design accuracy in the kinetic regime $\varepsilon=O(1)$ as well as the fluid regime $\varepsilon \ll 1$. This latter behavior is exactly due to the AP property of the methods. For completeness, we also report the results of the methods in the intermediate regime (i.e., $\varepsilon$ lies between 0 and 1), where the methods may exhibit some order reduction as expected. A careful study of this behavior is beyond the scope of the current work and left for future work.

\begin{rmk}
Note that the order conditions in Subsection \ref{sec:OC-IMEX}
do not guarantee that we will not observe order reduction.
When $\varepsilon = O(1)$ we expect to see the design accuracy predicted by the order conditions.
When $\varepsilon \ll 1$ design accuracy may not be evident due to the order reduction phenomenon.
However, the AP property allows us to recover full accuracy in the asymptotic limit $\varepsilon \rightarrow 0$.
\end{rmk}

\subsubsection{An ODE model}

We consider the ODE model \eqref{vdp} with
\begin{equation}
f(u_1) = 1+u_1^2,\quad g(u_1) = \sin u_1.
\end{equation}
We take the initial data as $u(0) = (2,0)^T$ (which is inconsistent initial data, i.e., we do not start from equilbrium), 
and solve \eqref{vdp} by the second and third order methods in Subsection \ref{sec:new_methods}, 
up to final time $T=1$, with various $\varepsilon$ and $\Delta t$. 
To calculate the error of a numerical solution $U=[U_1,U_2]^T$, we compare with a reference solution $U^{\text{ref}}$ obtained by the MATLAB solver \texttt{ode15s} with relative tolerance $\texttt{RelTol}=1e-13$ and absolute tolerance $\texttt{AbsTol}=1e-15$, and compute the error by
\begin{equation}
\text{error} = |U_1(T)-U^{\text{ref}}_1(T)| + |U_2(T)-U^{\text{ref}}_2(T)|.
\end{equation}
The results are shown in Figures \ref{ODE2nd} and \ref{ODE3rd}. 
For both methods, one can see the design order 
accuracy in the kinetic regime ($\varepsilon=O(1)$ and $\Delta t$ is relatively small) 
and the fluid regime ($\varepsilon\ll 1$ and $\Delta t$ is not very small), 
while in the intermediate regime (when $\varepsilon$ and $\Delta t$ are comparable) one can see some order reduction. In Figure \ref{ODE3rd} with $\varepsilon=1$ (and similar for $\varepsilon=0.01,1e-10$),  one can see that the error increases as $\Delta t$ decreases when $\Delta t$ is less than $5e-5$, 
and this is a consequence of the accumulation of round-off errors.

We note that the intermediate plateaus that are seen in Figures \ref{ODE2nd} and \ref{ODE3rd}  are not an indication  
of the order reduction phenomena that is usually observed in the AP literature, as we observe the errors are not converging 
at a rate of $O(\dt)$, but leveling off at the order of $\varepsilon$.  
This result is not caused by numerical round off  errors, as the schemes are still converging to a ``solution" 
at the designed order of accuracy.  Indeed, if we compare the solution at time-step $\dt$ to the $\dt/2$ solution, 
we observe design-order of convergence.
The explanation for  these $O(\varepsilon)$ plateaus  can likely be found by looking at the  
higher order asymptotic expansion. In practice, these errors are of  $O(\varepsilon)$
which are typically much smaller than other sources of errors in simulations thus not typically exhibited in practice.

\begin{figure}[hbt]
    \centering
    \begin{minipage}{0.475\textwidth}
        \centering
        \includegraphics[width=\textwidth]{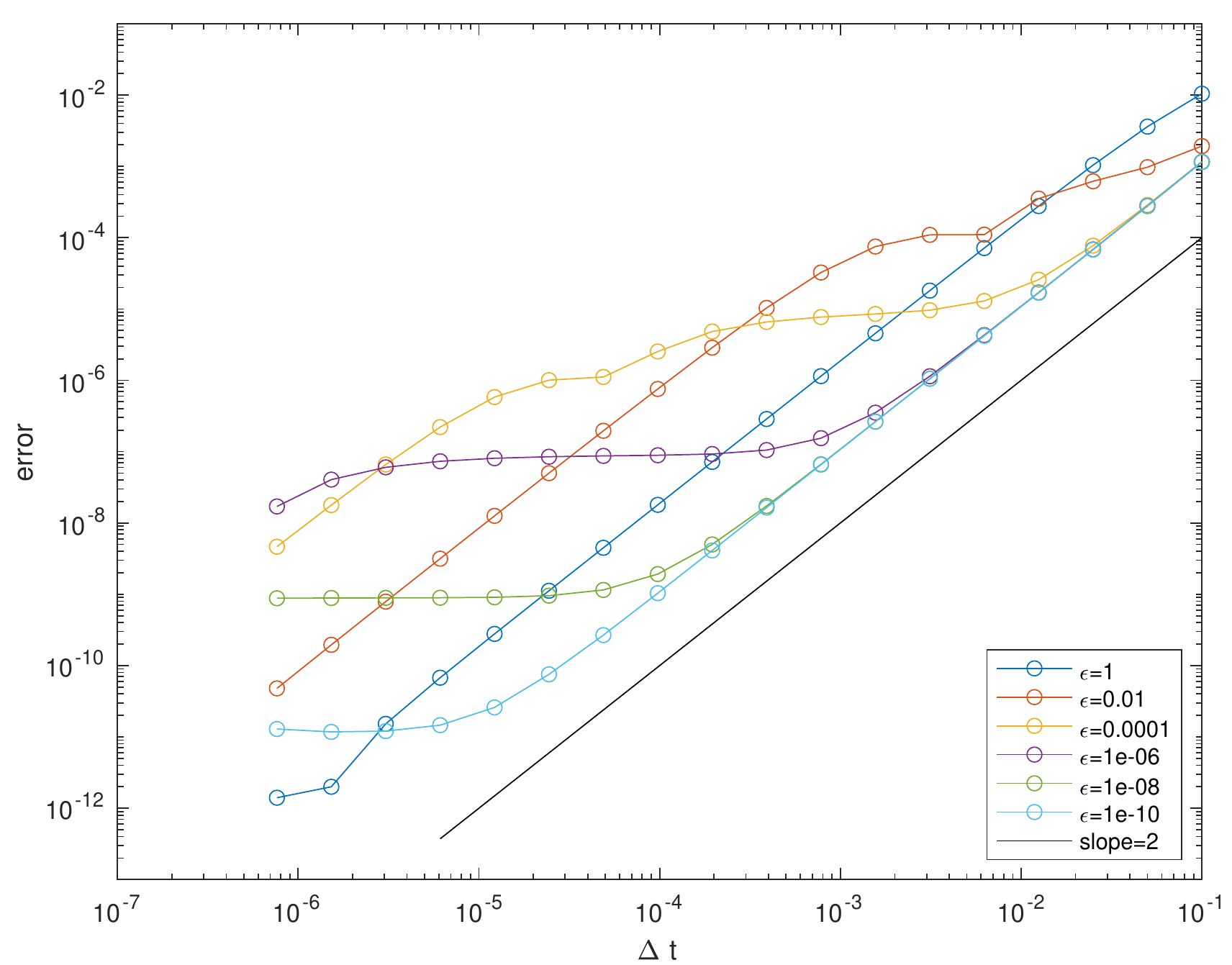} 
        \vspace{-.3in}
        \caption{Accuracy test of the new second order IMEX scheme for an ODE model. \label{ODE2nd}}
    \end{minipage}\hfill
    \begin{minipage}{0.475\textwidth}
        \centering
        \includegraphics[width=\textwidth]{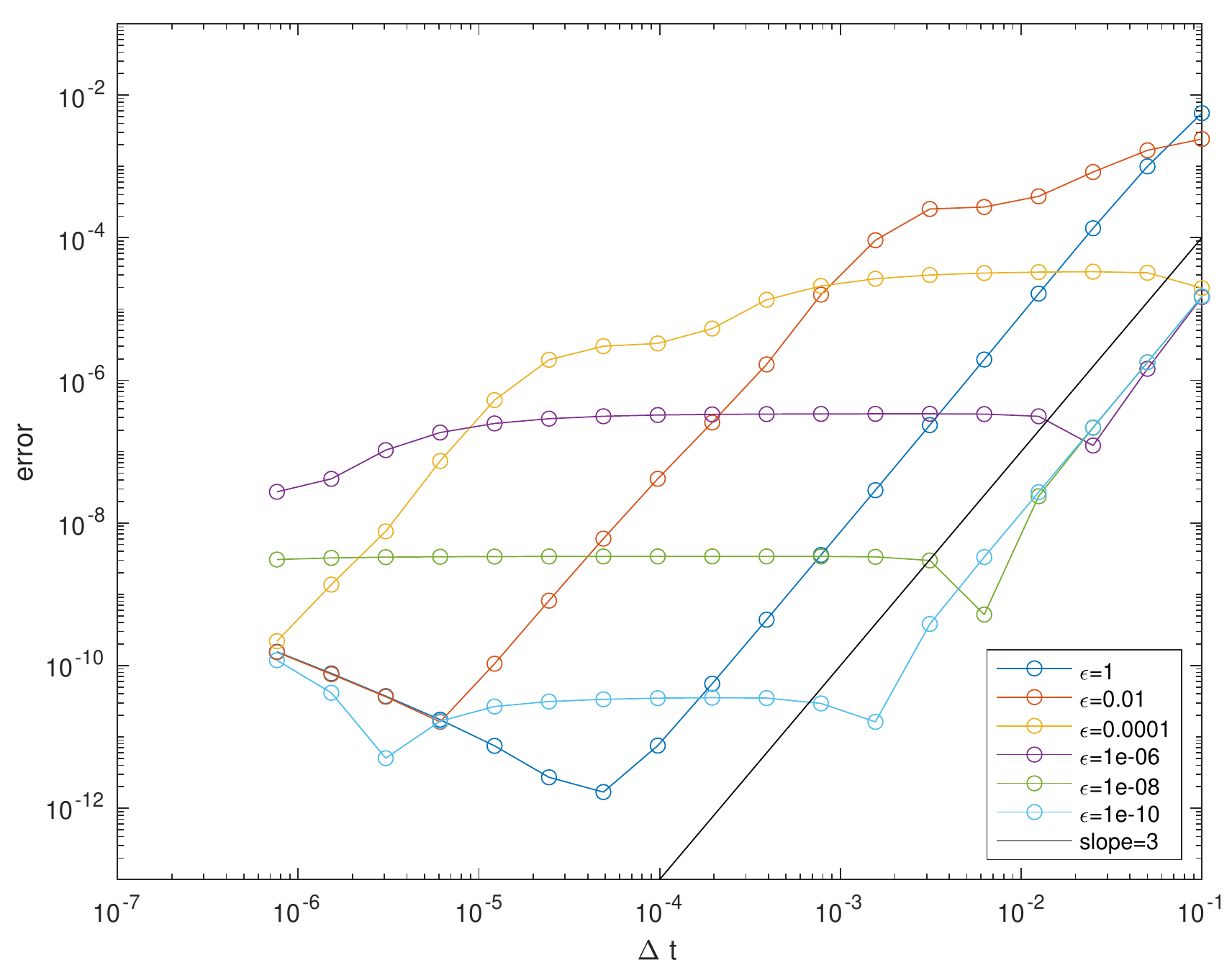} 
            \vspace{-.3in}
        \caption{Accuracy test of the new third order IMEX scheme for an ODE model.\label{ODE3rd}}
    \end{minipage}
\end{figure}

\subsubsection{The Broadwell model}

We consider the Broadwell model \eqref{broadwell} on the domain $x\in[0,2]$ with periodic boundary condition, with inconsistent initial data
\begin{eqnarray*}
f_+(0,\cdot) & =& 1+0.2 \exp(0.3\sin(\pi x)),  \quad  f_-(0,\cdot) = \exp(0.2\cos(2\pi x)),\\
 f_0(0,\cdot)&=&\frac{1}{1+0.3\sin(\pi x)}.
\end{eqnarray*}
We discretize in space by the fifth order finite volume WENO scheme, and the collision operator $Q$ is evaluated pointwise 
on the Gauss quadrature points in each cell, as described in  Subsection 3.3.2 of \cite{HSZ18}. 
We fix the the CFL number as $\Delta t = \frac{1}{2}\Delta x$, and solve \eqref{broadwell} by the second 
and third order methods in Subsection \ref{sec:new_methods} up to final time $T=0.1$. 
The error is computed by the $L^2$ norm of the difference between the numerical solution and one with a refined mesh. \rev{Note that in order for the fully discrete numerical scheme to be positivity-preserving, one has to use the positivity-preserving spatial discretization, for example, the positivity-preserving finite volume WENO scheme \cite{ZS10}, which requires a smaller CFL condition and a positivity-preserving limiter. Here since our main focus is to verify the order in time discretization and the AP property, we choose a larger time step and neglect the limiter.}

The results are shown in Figures \ref{Broadwell2nd} and \ref{Broadwell3rd}, and one can see similar behavior as in the previous subsection.

\begin{figure}[hbt]
    \centering
    \begin{minipage}{0.475\textwidth}
        \centering
        \includegraphics[width=\textwidth]{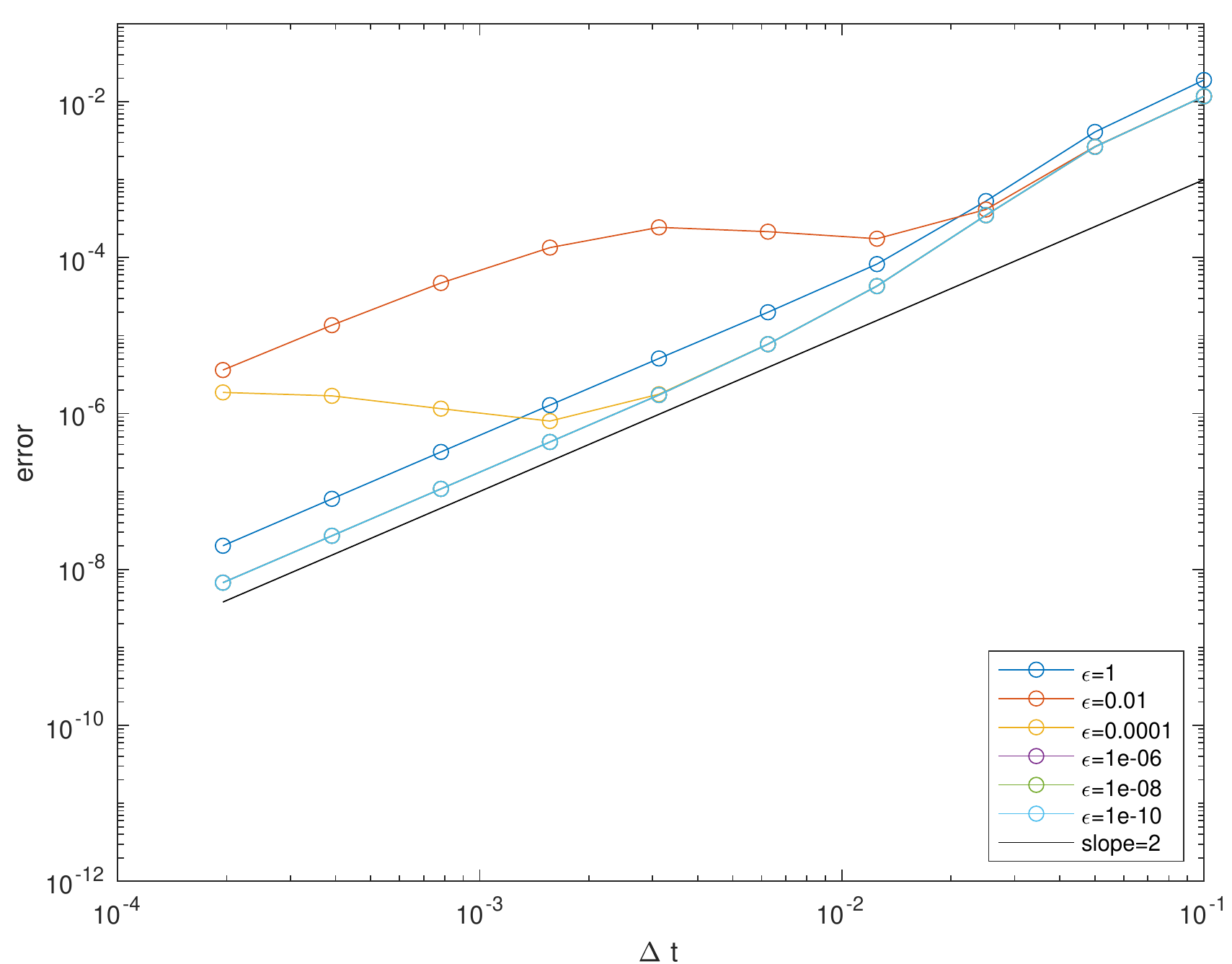} 
            \vspace{-.3in}
        \caption{Accuracy test of the new second order IMEX scheme for the Broadwell model. \label{Broadwell2nd}}
    \end{minipage}\hfill
    \begin{minipage}{0.475\textwidth}
        \centering
        \includegraphics[width=\textwidth]{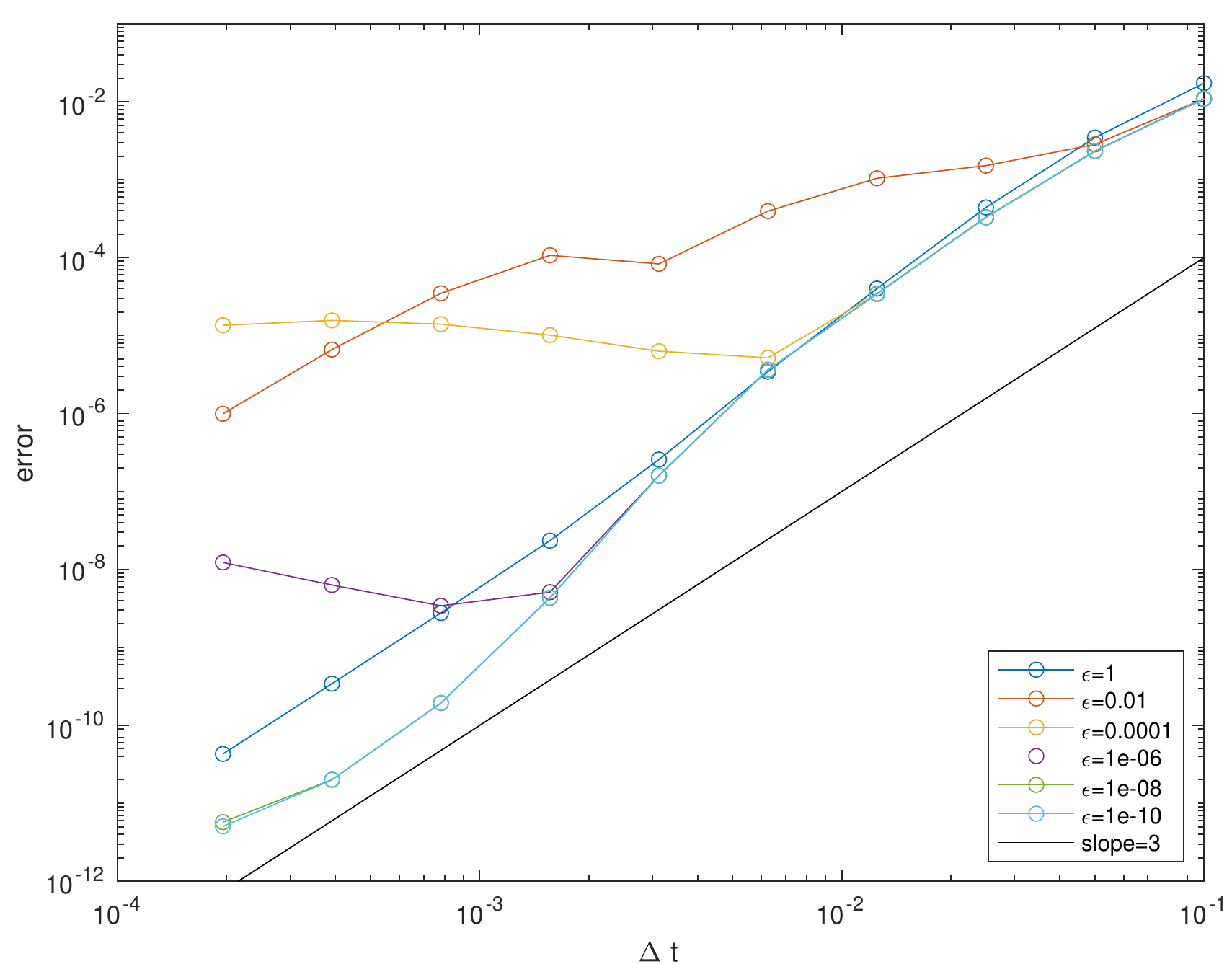} 
            \vspace{-.3in}
        \caption{Accuracy test of the new third order IMEX scheme for the Broadwell model.\label{Broadwell3rd}}
    \end{minipage}
\end{figure}

\subsubsection{The BGK model}

We consider the 1D BGK model \eqref{BGK} on the physical domain $x\in[0,2]$ with periodic boundary condition, 
and inconsistent initial data given by
\begin{equation}
f(0,x,v)=0.7M[\tilde{\rho}(x),\tilde{u}(x),\tilde{T}(x)](v) + 0.3M[\tilde{\rho}(x),-0.5\tilde{u}(x),\tilde{T}(x)](v),
\end{equation}
with
\begin{equation}
\tilde{\rho}(x) = 1+0.2\sin(2\pi x),\quad \tilde{u}(x)=1,\quad \tilde{T}(x) = \frac{1}{1+0.2\sin(\pi x)}.
\end{equation}
The velocity domain is truncated into $[-v_{max},v_{max}]$ with $v_{max}=15$ and discretized with $N_v=150$ grid points,
\begin{figure}[hbt]
    \centering
    \begin{minipage}{0.475\textwidth}
        \centering
        \includegraphics[width=\textwidth]{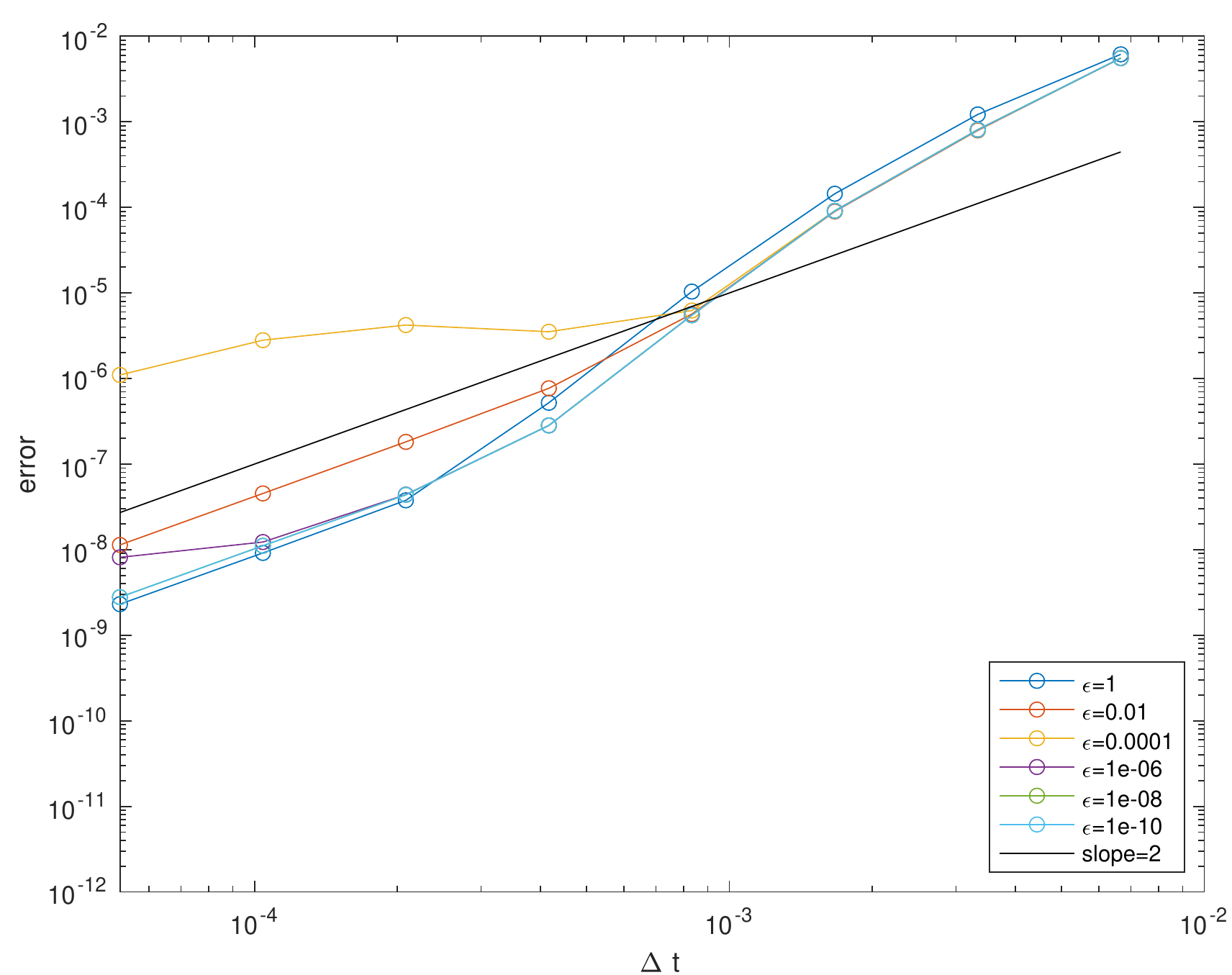} 
            \vspace{-.3in}
        \caption{Accuracy test of the new second order IMEX scheme for the BGK model. \label{BGK2nd}}
    \end{minipage}\hfill
    \begin{minipage}{0.475\textwidth}
        \centering
        \includegraphics[width=\textwidth]{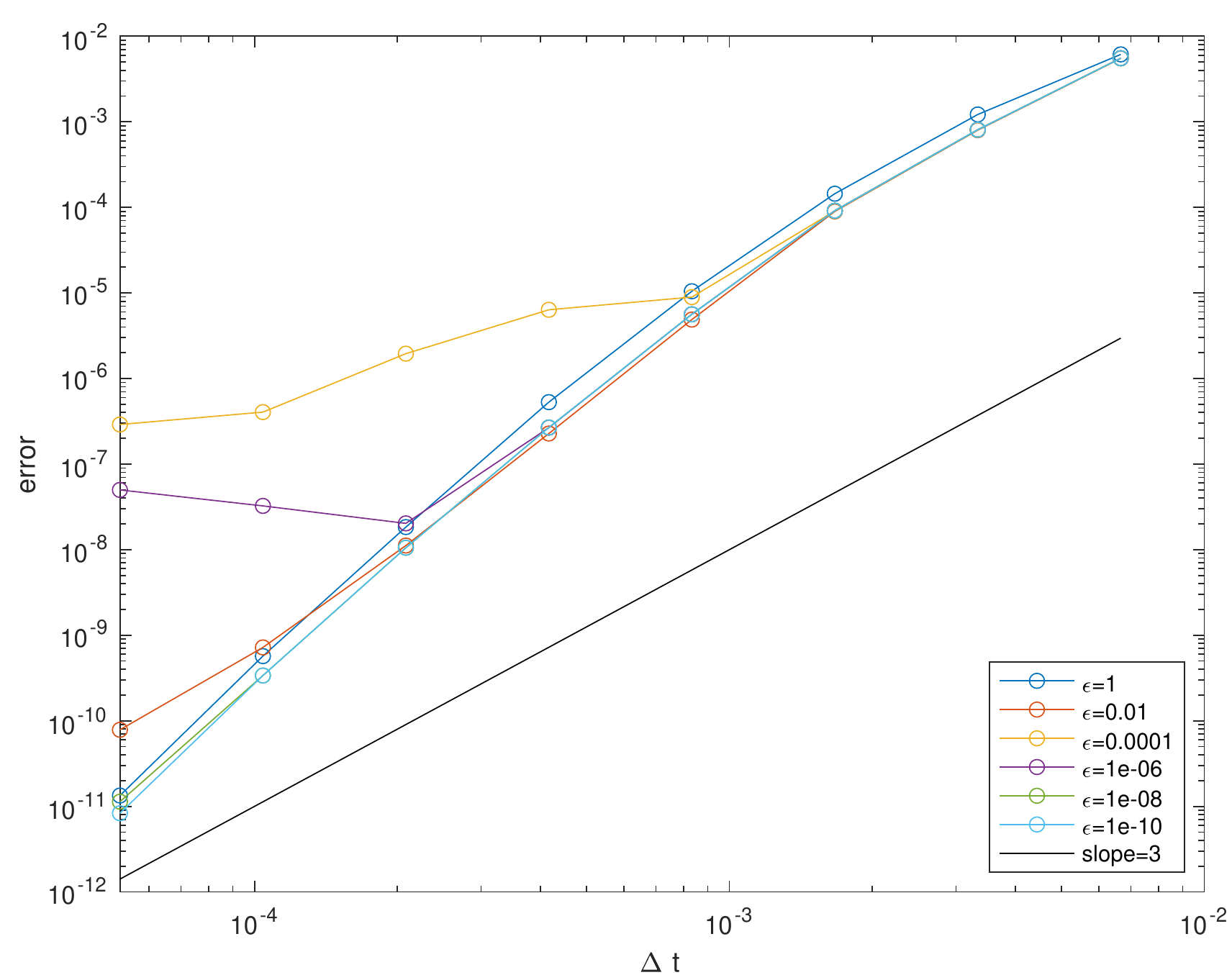} 
            \vspace{-.3in}
        \caption{Accuracy test of the new third order IMEX scheme for the BGK model. \label{BGK3rd}}
    \end{minipage}
\end{figure}
and the physical space is discretized in the same way as the previous subsection. We fix the the CFL number as $\Delta t = \frac{1}{2}\frac{\Delta x}{v_{max}}$, and solve \eqref{BGK} by the second and third order methods in Subsection \ref{sec:new_methods} up to final time $T=0.1$. The error is computed by the $L^2$ norm (in the $(x,v)$ space) of the difference between the numerical solution and one with a refined mesh. \rev{Note that the velocity space discretization may introduce some additional error such that the Properties 3 and 4 in Subsection \ref{sec:models} may not hold exactly. Here we chose a large velocity domain truncation and many grid points to make sure that the error from the velocity space discretization is negligible.}

The results are shown in Figures \ref{BGK2nd} and \ref{BGK3rd}. For the second order scheme, one can see clearly the second order accuracy when $\Delta t$ is small enough (so that the temporal error dominates) for both $\varepsilon=O(1)$ and $\varepsilon \ll1$, 
and order reduction is observed in the intermediate regime. For the third order scheme, 
when $\varepsilon=O(1)$ or $\varepsilon \ll 1$, the error 
converges at a higher than expected rate even for the smallest $\Delta t$ in the simulation, 
which suggests that the spatial error is still dominating. By comparing with the results of the second order scheme 
we see that the third order scheme indeed gives a much smaller error under the same time-step size.

\rev{Finally, to check the AP as well as the positivity-preserving properties, we use the second and third order multi-derivative IMEX methods in Subsection \ref{sec:new_methods} to solve a mixed regime problem, i.e., \eqref{BGK} with a variable Knudsen number $\varepsilon=\varepsilon(x)$ specified as below. This numerical example is comparable to the numerical result in Section 5.3 of \cite{HSZ18}.
}

\rev{
We take the physical domain as $x\in [0,2]$ with periodic boundary condition, and the variable Knudsen number 
\begin{equation}
\varepsilon(x) = \varepsilon_0 + (\tanh(1-11(x-1)) + \tanh(1+11(x-1))),\quad \varepsilon_0=10^{-5},
\end{equation}
so that the problem is in the kinetic regime ($\varepsilon(x)=O(1)$) near $x=1$, and in the fluid regime ($\varepsilon(x)\approx 10^{-5}$) for $x$ away from 1. The initial data is taken the same as eqs. (5.1)-(5.2) in \cite{HSZ18}. The final time is taken as $T=0.5$. For the new multi-derivative IMEX methods, we discretize the physical space by the fifth order finite volume WENO scheme with positivity-preserving limiters in \cite{ZS10}, and the velocity space is discretized in the same way as before. The variable Knudsen number is treated by a Gauss-Legendre quadrature in each spatial cell in the same way as Section 3.3.3 of \cite{HSZ18}. We take $N_x=40$ and $\Delta t=\frac{1}{24}\frac{\Delta x}{v_{max}}$ to satisfy the positivity-preserving CFL condition. 
}

\rev{In the simulation we tracked the numerical values (cell averages in the physical space) of $f$, and no negative cell is observed. The numerical solutions are compared with a reference solution obtained by the explicit second-order SSP-RK scheme with $N_x=80$ and $\Delta t = \frac{1}{240}\frac{\Delta x}{v_{max}} \approx 7\times 10^{-6}$, for which the smallest value of the Knudsen number (around $10^{-5}$) is resolved. The result is shown in Figure \ref{fig_mix}, in terms of the macroscopic quantities. One can see good agreement between the solution by the new schemes and the reference solution. This verifies the AP and positivity-preserving properties of the new multi-derivative IMEX methods.
}

\begin{figure}[hbt]
    \centering
    \begin{minipage}{0.49\textwidth}
        \centering
        \includegraphics[width=\textwidth]{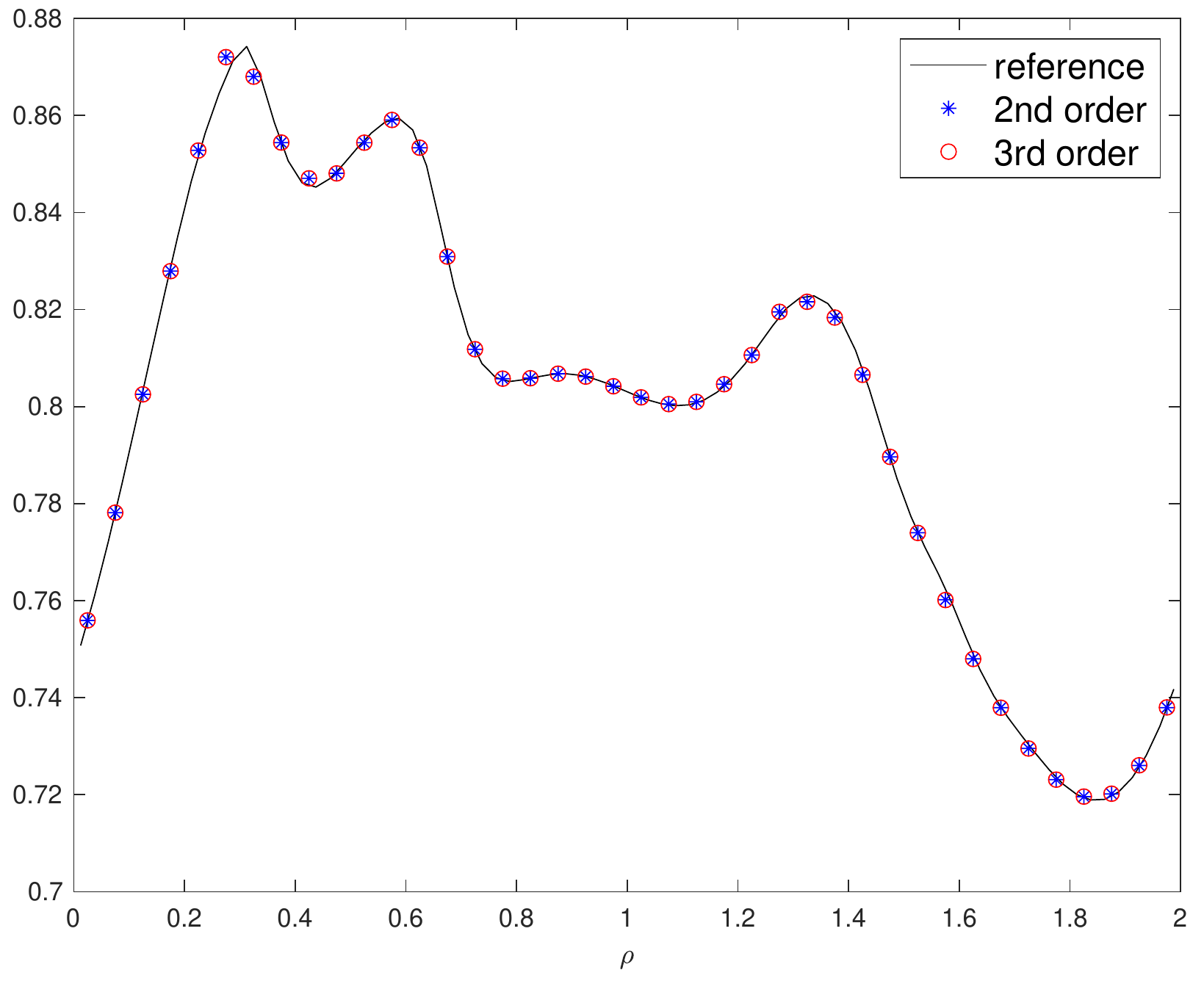} 
    \end{minipage}\hfill
    \begin{minipage}{0.49\textwidth}
        \centering
        \includegraphics[width=\textwidth]{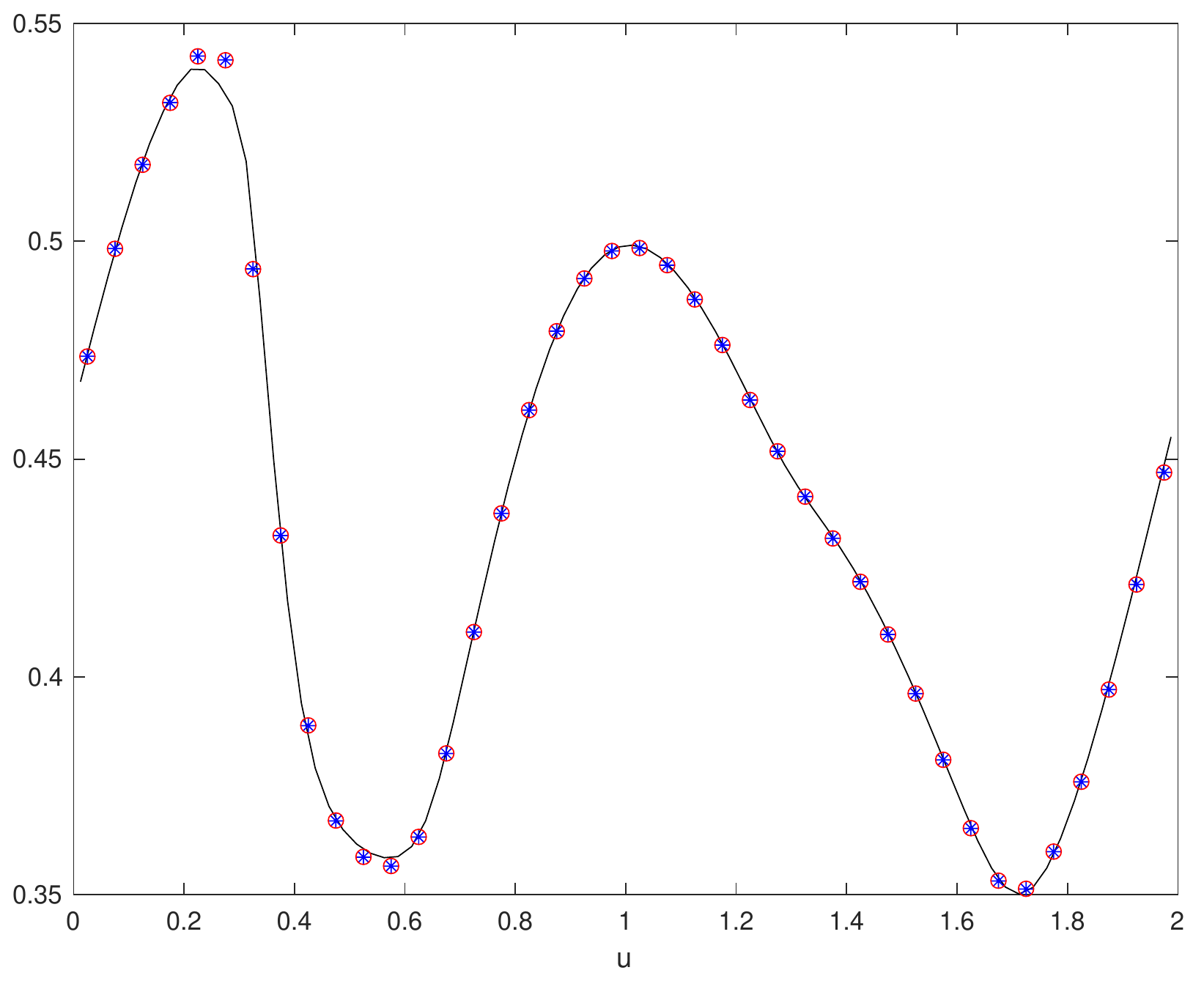} 
    \end{minipage}
    \begin{minipage}{0.49\textwidth}
        \centering
        \includegraphics[width=\textwidth]{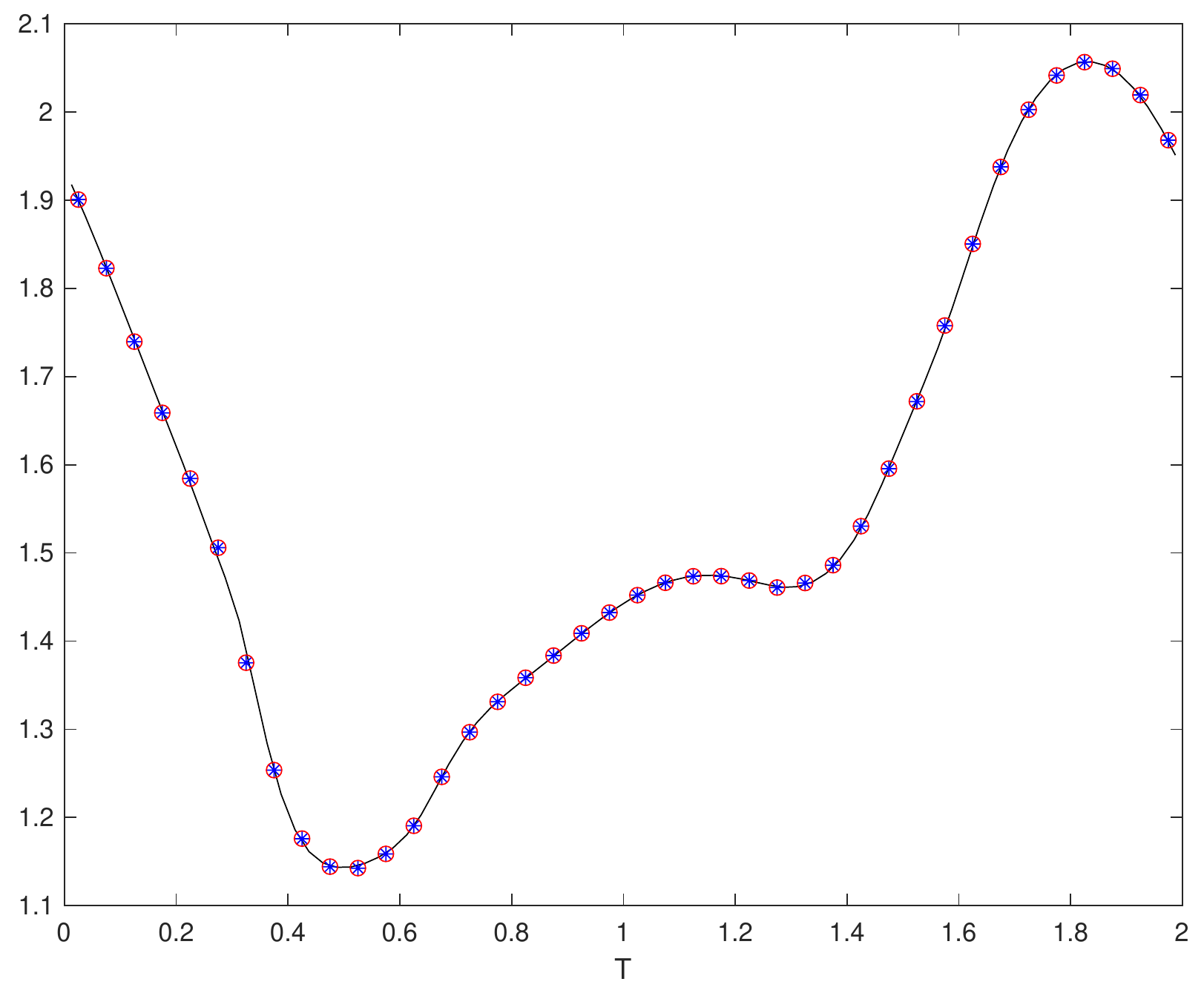} 
            \vspace{-.3in}
    \end{minipage}
    \caption{\rev{The mixed regime problem for the BGK model. Top left: density $\rho$; top right: bulk velocity $u$; bottom: temperature $T$. Asterisks/circles: numerical solutions by the new second/third order schemes. Solid line: the reference solution. The numerical solutions of the two new schemes are very close to each other because the spatial error is dominating.}}
    \label{fig_mix}
\end{figure}


%
%
\section{Conclusions}  \label{sec:conclusions}
In this work, we presented a class of unconditionally SSP implicit multi-derivative  Runge--Kutta schemes.
The unconditional SSP methods of order $p>2$ are novel, and is enabled by the backward derivative condition. 
This condition is an alternative to the second derivative conditions given in
\cite{SD,TS}, and  is highly relevant to a range of problems, as shown in Section \ref{sec:models}.  

The new backward derivative condition, which enabled the unconditionally SSP schemes, were inspired by the
 work in \cite{HSZ18} which derived positivity preserving and asymptotic preserving  
IMEX Runge--Kutta  methods with a derivative correction term.  We formulate 
multi-derivative  implicit-explicit (IMEX)  Runge--Kutta  methods that 
allow us to obtain order $p>2$   and to ensure that the method is  positivity preserving and asymptotic preserving when applied to 
problems that satisfy the five properties in Subsection \ref{sec:models}. 
In particular, we focus on an application area that includes
a hyperbolic relaxation model, the Broadwell model, and the BGK kinetic equation.
Such methods require treatment with an implicit-explicit (IMEX) time-stepping approach,
and it is desired that the method be AP and positivity preserving.

We derived and presented order conditions for SSP  IMEX multi-derivative Runge--Kutta methods,
and devised implicit methods that achieve fourth  order, and  IMEX methods that are third order,
and  are SSP under a time-step restriction independent of the stiff term.
The SSP condition ensures that the multi-derivative IMEX schemes are positivity preserving, and 
we present sufficient conditions  under which such methods are also 
asymptotic preserving  when applied to the problems of interest.
While we focused in the numerical examples on the IMEX schemes applied to  a hyperbolic relaxation system, 
the Broadwell model, and the BGK  equation, we stress that the results in this paper  are of broad use.
Any problems with operators that satisfy the forward Euler and -- if handled implicitly --
 the backward derivative condition can benefit from these
methods which are SSP with a time-step that does not depend on the function handled implicitly.

\newpage 
\appendix
\section{Unconditionally SSP implicit methods} \label{sec:appA}
\rev{
Previously, explicit SSP two-derivative methods were developed that preserved the 
 forward Euler \eqref{FEcond} and second derivative \eqref{SDcond} conditions \cite{SD} or the 
 forward Euler  \eqref{FEcond} and Taylor series  \eqref{TScond} conditions \cite{TS}.
 Methods that preserve the strong stability properties of these conditions require 
 nonnegative coefficients on the prior stages, the function,  and its derivative \cite{SD,TS}.
 In other words, we require that (elementwise)
 \begin{eqnarray} \label{forward-coef}
\mR \ve  \geq  0, \; \; \; \;  \mP  \geq  0, \; \; \;  \mD \geq  0, \; \; \; \; \dot{\mD} \geq  0.
 \end{eqnarray}
We show here that a method of the form \eqref{MDRK} that satisfies the conditions \eqref{forward-coef}
cannot be second order. This is simply a restatement of the proof in \cite{GST01} in the current notation.
}

\rev{
The first and second order conditions are
\[ \vb^T \ve = 1, \; \; \; \; \;  \vb^T \vc+ \dot{\vb}^T \ve=\frac{1}{2}.\]
Recall that $ \vb^T $ is the final row of $\mA $, and that $\vc$ is the row sum of $\mA$.
Note that the matrix
 $\mA = \mR^{-1} \mD = \left( I - \mP \right)^{-1} \mD$  can be written as
\[\mA= (I + \mP + \mP^2 +  . . . + \mP^{s-1}) \mD ,\]
a consequence of the fact that $\mP$ is strictly lower triangular and so $\mP^s$ becomes zero.
Let's look at each row of $\mA \ve$  and $\mA \vc$ using the recursive nature of the matrix multiplication: 
The first row is simply $ (\mA \ve)_1 = \mD_{11} $ and   $(\mA \vc)_1 = \mD_{11}^2$, the other rows are:
\[ (\mA \ve)_i = \mD_{ii} + \sum_{j=1}^{i-1} p_{ij}   \left(\mA \ve \right)_j \; \; \; \; \; 
(\mA \vc)_i = \mD_{ii} (\mA \ve)_i  + \sum_{j=1}^{i-1} p_{ij}   \left(\mA \vc \right)_j .\]
}

\rev{
For any real number $a$ the first row  satisfies
\[  (1-a) (\mA \ve)_1 - (\mA \vc)_1  =  (1-a) \mD_{11} - \mD_{11}^2 \leq k_1(1-a)^2\]
where $k_1 =\frac{1}{4}$.
We define 
\[k_i = \frac{1}{4 \left(1-k_{i-1}\right)},\]
and observe that $\frac{1}{4} = k_1 < k_2 < ... k_s < \frac{1}{2}$.
Now we can show recursively that if 
\[  (1-a) (\mA \ve)_j - (\mA \vc)_j \leq k_j (1-a)^2, \; \; \; \;  \forall j < i \] 
then  
\begin{align*}
(1-a)  (\mA \ve)_i - (\mA \vc)_i   & = (1-a)   \left( \mD_{ii} + \sum_{j=1}^{i-1} p_{ij}   \left(\mA \ve \right)_j \right) -
\left(  \mD_{ii} (\mA \ve)_i  + \sum_{j=1}^{i-1} p_{ij}   \left(\mA \vc \right)_j \right)  \\
& =  (1-a)    \mD_{ii}  -  \mD_{ii} (\mA \ve)_i  
+ \sum_{j=1}^{i-1} p_{ij}   \left( (1-a)  \left( \mA \ve \right)_j - \left(\mA \vc \right)_j  \right)    \\
& =   (1-a)   \mD_{ii}  -  \mD_{ii}^2   - \mD_{ii} \sum_{j=1}^{i-1} p_{ij}   \left(\mA \ve \right)_j 
+ \sum_{j=1}^{i-1} p_{ij}   \left( (1-a)  \left( \mA \ve \right)_j - \left(\mA \vc \right)_j  \right)    \\
& = (1-a)  \mD_{ii}  -  \mD_{ii}^2 
+ \sum_{j=1}^{i-1} p_{ij}  \left(  \left( 1-a  - \mD_{ii}  \right)  \left( \mA \ve \right)_j   - \left(\mA \vc \right)_j  \right)    \\
& < \left( 1-a -  \mD_{ii} \right)    \mD_{ii}  + k_{i-1} \left( 1-a  - \mD_{ii}  \right)^2.
\end{align*}
We look at this final  term and observe that it obtains a minimum at 
\[ \mD_{ii} = \frac{1}{2} \frac{(1-a)(2 k_{i-1} - 1)}{k_{i-1} -1}, \]
so that
\[
(1-a)  (\mA \ve)_i - (\mA \vc)_i    \leq \frac{1}{4(1-k_{i-1})}(1-a)^2 = k_i (1-a)^2.
\]
Using the value $a=0$ and looking at the final row $i=s$ we obtain
\[  \vb^T \ve - \vb^T \vc =  (\mA \ve)_s - (\mA \vc)_s    \leq k_s   < \frac{1}{2}.\]
If the method is at least first order, we must then have 
\[ \vb^T \vc >  \vb^T \ve - \frac{1}{2} = \frac{1}{2}.\]
We can then conclude that if 
\[\vb^T \vc+ \dot{\vb}^T \ve=\frac{1}{2}\]
and all the coefficients of $\mA$ are non-negative, then  $\dot{\vb}$ must have negative coefficients
or the method cannot be second order.
}

\rev{
This argument above shows that the conditions on the method lead to negative coefficients, and 
as both the forward Euler condition and either the second derivative or Taylor series condition
require positive coefficients on both the function and its derivative, the resulting method is not SSP.
Thus, implicit multi-derivative Runge--Kutta methods cannot be unconditionally SSP in the sense of
preserving the forward Euler and one of the derivative conditions above. This leads us
to consider the backward derivative condition.
}

\bibliography{hu_bibtex}

\end{document}